\documentclass[a4paper,leqno,11pt]{amsart}
\usepackage{amsmath,amsthm,amscd,amssymb,amsfonts, amsbsy, enumerate}
\usepackage{latexsym}
\usepackage{txfonts}
\usepackage[colorlinks,citecolor=red,pagebackref,hypertexnames=false]{hyperref}
\usepackage{pgf,tikz}
\usepackage{pdfsync}
\usepackage{marginnote}
\usepackage{exscale}

\usepackage{pgf}
\usepackage{color}

\calclayout
\allowdisplaybreaks

\numberwithin{equation}{section}

\theoremstyle{plain}
\newtheorem{theorem}[equation]{Theorem}
\newtheorem{lemma}[equation]{Lemma}
\newtheorem{corollary}[equation]{Corollary}
\newtheorem{proposition}[equation]{Proposition}

\theoremstyle{definition}

\theoremstyle{remark}
\newtheorem{remark}[equation]{Remark}

\newcommand{\fiint}{\operatornamewithlimits{\fint\!\!\!\!\fint}}
\newcommand{\dv}{\operatorname{div}}
\newcommand{\R}{\operatorname{Re}}
\newcommand{\supp}{\operatorname{supp}}

\newcommand{\tr}{\operatorname{Tr}}

\newcommand{\rn}{{\mathbb{R}^n}}
\newcommand{\co}{\mathbb{C}}
\newcommand{\reu}{{\mathbb{R}^{n+1}_+}}
\newcommand{\repm}{{\mathbb{R}^{n+1}_\pm}}
\newcommand{\ree}{{\mathbb{R}^{n+1}}}

\newcommand{\RR}{{\mathbb{R}}}
\newcommand{\NN}{{\mathbb{N}}}

\newcommand{\eps}{\varepsilon}

\newcommand{\p}{\mathcal{P}}
\newcommand{\s}{\mathcal{S}}
\newcommand{\A}{\mathcal{A}}
\newcommand{\C}{\mathcal{C}}

\newcommand{\N}{\widetilde{N}}

\newcommand{\bp}{\noindent {\it Proof}.\,\,}
\newcommand{\ep}{\hfill$\Box$ \vskip 0.08in}

\def\ring{\mathaccent"0017 }

\def\div{\mathop{\operatorname{div}}}

\begin{document}

\title[Regularity problem]{The Regularity problem for second order elliptic operators with complex-valued bounded measurable coefficients}

\author{Steve Hofmann}

\address{Steve Hofmann
\\
Department of Mathematics
\\
University of Missouri
\\
Columbia, MO 65211, USA} \email{hofmanns@missouri.edu}

\author{Carlos Kenig}

\address{Carlos Kenig
\\
Department of Mathematics
\\
University of Chicago
\\
Chicago, IL,  60637 USA} \email{cek@math.chicago.edu}

\author{Svitlana Mayboroda}

\address{Svitlana Mayboroda
\\
School of Mathematics
\\
University of Minnesota
\\
Minneapolis, MN, 55455  USA} \email{svitlana@math.umn.edu}

\author{Jill Pipher}

\address{Jill Pipher
\\
Department of Mathematics
\\
Brown University
\\
Providence, RI,   USA} \email{jpipher@math.brown.edu}

\thanks{Each of the authors was supported by NSF.  
\\
\indent
This work has been possible thanks to the support and hospitality of the \textit{University of Chicago},  the \textit{University of Minnesota}, the \textit{University of Missouri},  \textit{Brown University},  the \textit{Institute for Computational and Experimental Research in Mathematics}, and the \textit{American Institute of Mathematics}. The authors would like to express their gratitude to these institutions.}

\maketitle

\begin{abstract}  The present paper establishes a certain duality between the Dirichlet  and  Regularity problems for elliptic operators with $t$-independent complex bounded measurable coefficients ($t$ being  the transversal direction to the boundary). To be precise, we show that  the Dirichlet boundary value problem is solvable in $L^{p'}$, subject to the square function and  non-tangential maximal function estimates, if and only if the corresponding Regularity problem is solvable in $L^p$. Moreover, the solutions admit layer potential representations.

In particular, we prove that for any elliptic operator with  $t$-independent  real (possibly non-symmetric)  coefficients there exists a $p>1$ such that the Regularity problem is well-posed in $L^p$.  
  \end{abstract}


\section{Introduction}

We consider a divergence form elliptic operator
$$L:=-\dv A(x)\nabla,$$
defined in $\mathbb{R}^{n+1}=\{(x,t), \,x\in\RR^n,\,t>0\}$. 
Here $A$ is an $(n+1)\times(n+1)$ matrix of bounded, complex-valued, $t$-independent coefficients,  
which satisfies the 
uniform ellipticity condition
\begin{equation}
\label{eq1.1} \lambda|\xi|^{2}\leq\,\R\, \langle A(x)\,\xi,\xi\rangle
:= \R\,\sum_{i,j=1}^{n+1}A_{ij}(x)\,\xi_{j} \bar \xi_{i}, \quad
  \Vert A\Vert_{L^{\infty}(\mathbb{R}^{n})}\leq\lambda^{-1},
\end{equation}
 for some $\lambda>0$, and for all $\xi\in\co^{n+1}$, $x\in\mathbb{R}^{n}$. The number $\lambda$ in \eqref{eq1.1} will be referred as the ellipticity parameter of $L$.  
As usual, the divergence form equation is interpreted in the weak sense, i.e., we say that $Lu=0$
in a domain $\Omega$ if $u\in W^{1,2}_{loc}(\Omega)$ and 
\begin{equation}\label{eq1.2}\iint_{\Omega} A \nabla u \cdot \overline{\nabla \Psi} = 0\,,
\end{equation}
\noindent for all  complex valued $\Psi \in C_0^\infty(\Omega)$.   
For us, $\Omega$ will be a Lipschitz graph domain
\begin{equation}\label{eq1.3}
\Omega_\psi:=\{(x,t)\in\ree: t>\psi(x)\}\,,
\end{equation}
where $\psi:\rn\to\RR$ is a Lipschitz function, 
or more specifically (but without loss of generality), $\Omega$ will be the half-space
$\reu:=\{(x,t)\in\mathbb{R}^{n}\times(0,\infty)\}$. We shall return to this point below.

Let us start by defining the weak solutions to the Dirichlet problem for the case of the nice data, that is, consider $Lu=0$ in $\reu$, $u\Big|_{\RR^n}=f \in C_0^\infty(\rn)$. In principle, taking a harmonic extension of $f$ to $\reu$ (denote it by $w$) and then using the Lax-Milgram lemma to resolve $Lu=-Lw$ with zero trace on the boundary, we get a solution in $\dot W^{1,2}(\reu)$, the factor space of functions modulo constants with the seminorm given by the norm of the gradient in $L^2(\reu)$. It is somewhat more convenient though to use a non-homogeneous space. One option (and here  
we follow an approach in \cite{KR}) is to work in the following framework. Let $\widetilde W^{1,2}(\reu)$ denote the space of functions $F$ for which 
$$ \|F\|_{\widetilde W^{1,2}(\reu)}:=\left(\iint_{\reu} |F(X)|^2\,\frac{dX}{1+|X|^2} \,+\, \iint_\reu |\nabla F(X)|^2\,dX\right)^{1/2}<\infty.$$ One can define the trace operator, for instance, as $\rm{Tr} : \widetilde{W}^{1,2}(\reu) \to \widetilde L^2(\RR^n)$, a continuous extension  of the restriction to the boundary operator, with $\widetilde L^2(\RR^n)$ denoting the space of functions $f$ on $\RR^n$ with 
$$\|f\|_{\widetilde L^2(\RR^n)}=\left(\int_{\RR^n}|f(x)|^2\,\frac{dx}{1+|x|}\right)^{1/2}<\infty, $$
\noindent (follow, e.g., the argument in \cite{E}, p. 272).  Notation $\widetilde  W^{1,2}_0(\reu)$ stands for  the space of functions in $\widetilde  W^{1,2}(\reu)$ with trace zero. In Lemma~\ref{l2.4} below we present a detailed argument showing that, in particular, for every $f\in C_0^\infty(\rn)$ there exists a unique solution $u\in \widetilde W^{1,2}(\reu)$ to the boundary problem 
\begin{equation}
\begin{cases} Lu=0\text{ in }\mathbb{R}_{+}^{n+1}\\ 
\lim_{t\to 0}u =f,
\end{cases}\label{eq1.0}\end{equation} 
where $\lim_{t\to 0}u$ is interpreted in the sense of the trace operator as discussed above. This solution will be referred to as the {\it weak} solution hereafter.

We say that the Dirichlet problem \eqref{Dp'} for $L$ is solvable for some $1<p'<\infty$ if for every $f\in C_0^\infty(\rn)$ the weak solution to the boundary problem \eqref{eq1.0} satisfies the non-tangential maximal function estimate 
\begin{equation}\label{eq1.4}
\|N_*(u)\|_{L^{p'}(\mathbb{R}^n)}\leq C \|f\|_{L^{p'}(\rn)}.
\end{equation}
Here,
$$N_* F(x)  \equiv \sup_{(z,t)\in \Gamma(x)}|F(z,t)|, $$ 
with $\Gamma(x):
=\{(y,t)\in\mathbb{R}_{+}^{n+1}:|y-x|<t\}$. Respectively, we write that 
\begin{equation}
\begin{cases} Lu=0\text{ in }\mathbb{R}_{+}^{n+1}\\ 
\lim_{t\to 0}u =f,\\ 
\|N_*(u)\|_{L^{p'}(\mathbb{R}^n)}<\infty\,,
\end{cases}\tag{$D_{p'}$}\label{Dp'}\end{equation}
is solvable. Note that, modulo some necessary explanations of the essence of the weak solution caused by the generality of $L$ at hand and provided above and in Section~\ref{sPrelim},  this definition coincides with the one classically used in this context (see, e.g., \cite{K}).

We say that the Dirichlet boundary value problem is well-posed if for every $f\in L^{p'}(\rn)$ there exists a unique solution to \eqref{Dp'} satisfying \eqref{eq1.4}, with $\lim_{t\to 0}u =f$ interpreted in the sense that $u(\cdot, t)$ converges to $f$ as $t\to 0$ strongly in $L^{p'}(\rn)$, and if moreover, the corresponding solution in the special case of $f\in C_0^\infty(\rn)$ coincides with the weak solution defined above in $\widetilde W^{1,2}(\reu)$.

We say that the Regularity problem \eqref{Rp} for $L$ is solvable for some $1<p<\infty$ if for every $f\in C_0^\infty(\rn)$ the weak solution to the boundary problem \eqref{eq1.0} satisfies the non-tangential maximal function estimate
\begin{equation}\label{eq1.5}
 \|\N(\nabla u)\|_{L^p(\RR^n)}\leq C \| \nabla_\| f\|_{L^p(\RR^n)},
 \end{equation}
where the modified non-tangential maximal function is given by
$$ \widetilde NF(x)  \equiv\sup_{(z,t)\in \Gamma(x)}\left(\fiint_{W(z, t)}|F(y,s)|^{2}dyds\right)^{\frac{1}{2}},$$
and   $W(x,t) \equiv \Delta(x,t) \times (t/2, 3t/2)$, $\Delta(x,t) \equiv \left\{y \in \rn: \left| x-y \right| < t\right\}$.  Respectively, we write that 
 \begin{equation}
\begin{cases} Lu=0\text{ in }\mathbb{R}_{+}^{n+1}\\ 
\lim_{t\to 0}u =f,\\ 
\|\widetilde N (\nabla u)\|_{L^p(\mathbb{R}^n)}<\infty\,,
\end{cases}\tag{$R_p$}\label{Rp}\end{equation}
is solvable. 

Let us note that the condition $\|\widetilde N (\nabla u)\|_{L^p(\mathbb{R}^n)}<\infty$ imposed on the solution above automatically implies that there exists a non-tangential trace, or, more precisely, $\lim_{t\to 0}u =f \, {\rm n.t.}$ (that is, $\lim_{(y,t) \to (x,0)} u(y,t) = f(x),$ for $a.e. \, x \in \mathbb{R}^n$) and  $\nabla_\| u(\cdot, t)$ converges to $\nabla_\| f$ as $t\to 0$ weakly in $L^{p}(\rn)$ (see Lemma~\ref{l2.2} below). Thus, for $f\in C_0^\infty$ we can interpret  $\lim_{t\to 0}u =f$ as a trace operator acting on  $\widetilde W^{1,2}(\reu)$ or as a non-tangential trace and the two traces coincide (see also \cite{BM}, Remark~7.13,  for a detailed discussion).

We say that the Regularity boundary value problem is well-posed if for every $f\in \dot L^p_1(\RR^n)$ there exists a unique solution to \eqref{Rp} satisfying \eqref{eq1.5}, with $\lim_{t\to 0}u =f$ interpreted in the sense that $\lim_{t\to 0}u =f \, {\rm n.t.}$ and  $\nabla_\| u(\cdot, t)$ converges to $\nabla_\| f$ as $t\to 0$ weakly in $L^{p}(\rn)$, and if moreover, the corresponding solution in the special case of $f\in C_0^\infty(\rn)$ coincides with the weak solution defined above in $\widetilde W^{1,2}(\reu)$.

The homogeneous Sobolev space $\dot L^p_1(\RR^n)$ is the completion of $C_0^\infty$ with respect to the Sobolev norm $\|\nabla_\| f\|_{L^p(\RR^n)}$. While fairly evident here, it will be convenient to distinguish the gradient in $\ree$ and the gradient in $\rn$ throughout the paper, and we shall denote the latter by $\nabla_\|.$

We note, furthermore, that \eqref{Dp'} and \eqref{Rp} above are defined in $\reu$. An analogous definition applies to the lower half-space. Well-posedness in $\RR^{n+1}_\pm$ stands for the well posedness both in $\reu$ and $\RR^{n+1}_-$, and similarly for other properties. 

 Let us comment on  a somewhat peculiar definition of well-posedness in this paper, insisting that for $C_0^\infty(\rn)$ data the solution coincides with the weak solution defined above in $\widetilde W^{1,2}(\reu)$. The rationale for such a definition comes, in particular, from the example in \cite{KKPT}, \cite{KR}, \cite{Ax}, where it has been demonstrated that the solution of \eqref{Dp'} in principle does not have to coincide with the weak solution, even for nice data.

Indeed, the papers \cite{KKPT}, \cite{KR}, \cite{Ax} consider solvability of boundary-value problems for the (two-dimensional) coefficient matrix
\[A_k(x,t)=A_k(x)=\begin{pmatrix}1&k\mathop{\mathrm{sgn}}(x)\\ -k\mathop{\mathrm{sgn}}(x)&1\end{pmatrix}\]
where $k$ is a real number.  It turns out that for certain values of $k$ and~$p$, the Dirichlet problem is solvable in the sense that for every $f\in L^{p'}(\rn)$ there exists a solution satisfying \eqref{Dp'} and converging to $f$ as $t\to 0$ in the strong $L^{p'}$ sense \cite{Ax} (moreover, according to  \cite{HMM}, such a solution is unique), but however, the weak solution (with nice datum which, in particular, belongs to $L^{p'}(\rn)$) does not necessarily satisfy  the non-tangential maximal function estimate in \eqref{Dp'} \cite{KKPT}. And indeed, the solution built in \cite{Ax} does not satisfy $\nabla u\in L^2(\RR^2_+)$, not even for smooth boundary data. Our definitions are aimed to avoid such a situation.

It has been proved in \cite{HKMP} that for any elliptic operator $L$ with real bounded measurable $t$-independent coefficients there exists $p'<\infty$ such that the Dirichlet problem \eqref{Dp'} is well-posed. The purpose of this paper is to establish a certain duality between the Dirichlet and the Regulatity problems, and in particular, to show that for any elliptic operator $L$ with real bounded measurable $t$-independent coefficients there exists a $p>1$ such that the Regularity problem \eqref{Rp} is well-posed.  Before stating the main result, let us introduce some relevant terminology. 

Here and throughout the paper,  the capital letters $X,Y,Z$  denote points in $\ree$ and the corresponding small ones stand for the points in $\rn$. Furthermore, 
$B=B_R(X)=B(X,R)$ is the ball in $\ree$ centered at $X\in \ree$ with the radius $R>0$, and $\Delta=\Delta_R(x)=\Delta(x,R)$ is the ball in $\rn$ centered at $x\in\rn$ with the radius $R>0$. Then the tent regions are $T(\Delta)=\widehat{\Delta}=\{(x,t)\in\reu: \text{dist}(x,\Delta^c)\geq t\}$, and the Whitney cubes are, as above, $W(x,t) =\Delta(x,t) \times (t/2, 3t/2)$, $(x,t)\in\reu$.

Throughout the paper $L$ will be an elliptic  divergence form elliptic operator
with  bounded, measurable, complex-valued, $t$-independent coefficients. We shall assume, in addition, that the solutions to $Lu=0$ in $\reu$ are locally H\"older continuous in the following sense.  
Assume that $Lu=0$ in $\reu$ in the weak sense and $B_{2R}(X)\subset \reu$, $X\in \reu$, $R>0$. Then 
\begin{equation}
|u(Y)-u(Z)|\leq
C \left(\frac{|Y-Z|}{R}\right)^{{\mu}}\left(\,\fiint\limits_{B_{2R}(X)}|u|^{2}\right)^{\frac{1}{2}},\quad\mbox{for all}\quad Y,Z \in B_R(X),
\label{eq1.6}\end{equation}
\noindent for some constants $\mu>0$ and $C>0$.
 In particular, one can show that for any $p>0$
\begin{equation}
|u(Y)|\leq
C \left(\,\fiint\limits_{B_{2R}(X)}|u|^{p}\right)^{\frac{1}{p}},\quad\mbox{for all}\quad Y,Z \in B_R(X).
\label{eq1.7}\end{equation}

We shall refer to property \eqref{eq1.6} by saying that the solutions (or, slightly abusing the terminology, the operator) satisfy the De Giorgi-Nash-Moser (or DG/N/M)
bounds.  Respectively, the constants $C$ and $\mu$ in \eqref{eq1.6}, \eqref{eq1.7} will be referred to as the De Giorgi-Nash-Moser constants of $L$. Finally, following \cite{AAAHK}, \cite{HMiMo}, \cite{HMM}, we shall normally refer to the following collection of quantities: the dimension, the ellipticity, and the De Giorgi-Nash-Moser constants of $L$, $L^*$ collectively as the {\it ``standard constants"}.

We note that the De Giorgi-Nash-Moser bounds are not necessarily satisfied for all elliptic PDEs with complex $t$-independent 
coefficients \cite{F, MNP, HMMc}. However, \eqref{eq1.6}, \eqref{eq1.7} always hold when the coefficients of the underlying equation are real \cite{DeG,Na,Mo}, 
 and the constants depend quantitatively only upon ellipticity and dimension
(for this result,  the matrix $A$ need not be $t$-independent).
Moreover, \eqref{eq1.6} (which implies \eqref{eq1.7}) is stable under small complex perturbations
of the coefficients in the $L^\infty$ norm (see, e.g., \cite{Gi}, Chapter VI, or  \cite{A1}).  
Thus, in particular, \eqref{eq1.6}-\eqref{eq1.7} hold automatically, e.g.,  for small complex
perturbations of real elliptic coefficients.  We also note that in the $t$-independent setting that we consider here, the interior the De Giorgi-Nash-Moser estimates hold always when the ambient dimension $n+1=3$
(see \cite[Section 11]{AAAHK}).

Let us now turn to the layer potentials. Let $L$  be an elliptic operator with bounded measurable coefficients. By  $E,\,E^*$  we denote the fundamental solutions associated with $L$ and $L^*$, respectively, in $\ree$, so that 
$$L_{x,t} \,E (x,t;y,s) = \delta_{(y,s)}(x,t) \;\;\;\text{and}\;\;\; L^*_{y,s}\, E^*(y,s;x,t) \equiv L^*_{y,s} \,{\overline{E (x,t;y,s)}} = \delta_{(x,t)}(y,s),$$
where $\delta_{(x,t)}$ denotes the Dirac delta function at the point $(x,t)$. 
One can refer, e.g., to \cite{HK} for their construction and properties. 
We note for future reference that when the coefficients of the underlying matrix are $t$-independent, 
\begin{equation}\label{eq1.9}
E(x,t;y,s)=E(x,t-s;y,0),
\end{equation}

\noindent and hence, in particular, one can swap the derivatives in $t$ and $s$ for the fundamental solution. 

The single layer potential and the double layer potential operators associated with $L$ are given, respectively, by 
\begin{multline}\label{eq1.10}
{S}^L_t f(x)  \equiv\int\limits_{\mathbb{R}^{n}}E(x,t;y,0)\,f(y)\,dy, \,\,\, t\in \mathbb{R},\,x\in\rn,\\
D^L_{t}f(x) 
\equiv\int\limits_{\mathbb{R}^{n}} {\overline{\partial_{\nu_{A^*},y}\, E^*
(y,0;x,t)}}\,f(y)\,dy,\,\,\, t \neq 0,\,x\in\rn.\end{multline}
Here, the conormal derivative is roughly $\partial_{\nu_A} u=-e_{n+1}A(y)\nabla u$, $e_{n+1}=(0,...,0,1)$. The precise meaning of the latter on the boundary will be discussed later, with the Preliminaries.

The main result of this paper is as follows. 

\begin{theorem}\label{t1.11} Let $L$ be a  divergence form elliptic operator
with  bounded, complex-valued, $t$-independent coefficients in $\repm$, $n\geq 2$. Assume, in addition, that the solutions to $Lu=0$ in $\RR^{n+1}_\pm$ satisfy the De Giorgi-Nash-Moser bounds, and that the same is true for the adjoint operator $L^*$. Then there exists $\eps>0$ depending on the standard constants only such that for any $1<p<2+\eps$, $\frac 1p+\frac {1}{p'}=1$, the  following are equivalent:
\begin{enumerate}[(a)]
\item The Dirichlet problem \eqref{Dp'} for $L^*$ is solvable in $\RR^{n+1}_\pm$, and, in addition to \eqref{eq1.4}, the weak solution satisfies the square function bounds
\begin{equation}\label{eq1.12}\|\A(t\nabla u)\|_{L^{p'}(\rn)}\leq C \|f\|_{L^{p'}(\rn)}, 
\end{equation}
where  $\A$ stands for the square function, that is,
\begin{equation}\label{eq1.13}
\A F(x):=\left(\iint_{|x-y|< |t|}|F(y,t)|^2 \frac{dy dt}{|t|^{n+1}}\right)^{1/2},
\end{equation}
\noindent for $F:\repm\to \RR$.
\item The Regularity problem \eqref{Rp} for $L$ is solvable in $\RR^{n+1}_\pm$. 
\item The Regularity problem \eqref{Rp} for $L$ is solvable in $\RR^{n+1}_\pm$, and the solution can be represented by means of (compatible) layer potentials, that is, 
\begin{equation}\label{eq1.14} u(x,t)=S_t^L \left(S_0^L\right)^{-1}f(x), \quad (x,t)\in \RR^{n+1}_\pm.\end{equation}
In particular, the operator $f\mapsto \N(\nabla S_{t}^Lf)$ is bounded in $L^p(\rn)$ and
$$ S_0^L:L^p(\rn)\to \dot L_1^p(\rn)$$
\noindent is compatibly invertible.
\item The Dirichlet problem \eqref{Dp'} for $L^*$ is solvable in $\RR^{n+1}_\pm$, and the solution can be represented by means of (compatible) layer potentials as
\begin{equation}\label{eq1.15} u(x,t)=S_t^{L^*} \left(S_0^{L^*}\right)^{-1}f(x), \quad (x,t)\in \RR^{n+1}_\pm.\end{equation}
In particular, $f\mapsto N(S_{t}^{L^*}f)$ is a bounded operator from $L_{-1}^{p'}(\rn)$ to $L^{p'}(\rn)$ and
$$ S_0^L: L_{-1}^{p'}(\rn)\to  L^{p'}(\rn)$$
\noindent is compatibly invertible. Here $L_{-1}^{p'}(\rn):= \left(\dot L_1^p(\rn)\right)^*$.
\item The Dirichlet problem \eqref{Dp'} for $L^*$ is well-posed in $\RR^{n+1}_\pm$, and, in addition to \eqref{eq1.4}, the solution satisfies the square function bounds \eqref{eq1.12}.
\item The Regularity problem \eqref{Rp} for $L$ is well-posed in $\RR^{n+1}_\pm$.
\end{enumerate}

\end{theorem}

\begin{remark}\label{r6.5.1} Let us comment on what is meant by compatibility (or rather compatible invertibility) of layer potentials. 

It is known that $(S^L)^\pm: \dot L^2_{-1/2}(\rn)\to \dot L^2_{1/2}(\rn)$  is  an invertible operator, essentially by the Lax-Milgram lemma and suitable trace/extension theorems (see, e.g., \cite{AMM}, Section 13). It is, however, possible, that the two inverses of $(S^L)^\pm$, one in $\dot L_1^p(\rn)\to L^p(\rn)$ (or, respectively, $L^{p'}(\rn)\to L_{-1}^{p'}(\rn)$) and another one $\dot L^2_{1/2}(\rn)\to \dot L^2_{-1/2}(\rn)$ are not compatible, that is, when acting on $f\in \dot L^2_{1/2}(\rn)\cap \dot L^p_1(\rn)$ (or, respectively, on $f\in \dot L^2_{1/2}(\rn)\cap L^{p'}(\rn)$)  they produce different functions. This is due to the aforementioned counterexample in \cite{Ax} (see also the corresponding discussion in \cite{BM}). 

It is important that this is not the case here. That is, the inverse of $(S^L)^\pm: L^p(\rn)\to \dot L^p_1(\rn)$ in statement $(c)$ (and, respectively, the inverse of $ S_0^L: L_{-1}^{p'}(\rn)\to  L^{p'}(\rn)$ is statement $(d)$) is compatible with the inverse of $(S^L)^\pm: \dot L^2_{-1/2}(\rn)\to \dot L^2_{1/2}(\rn)$. We shall refer to this property as {\it compatible} invertibility hereafter, the corresponding inverses will be called compatible, and we shall say that the representations \eqref{eq1.14}, \eqref{eq1.15} feature compatible layer potentials (while referring to the involved inverses).

\end{remark}


In combination with the results in \cite{HKMP}, \cite{AAAHK}, \cite{HMiMo},
and \cite{R}  (as regards the latter, see also \cite{GH} for an alternative proof), 
Theorem~\ref{t1.11} yields the following Corollaries. Their proofs can be found in Section~\ref{sPrCor}.

\begin{corollary}\label{c1.16} Let $L$ be a divergence form elliptic operator with bounded measurable $t$-independent coefficients, and assume that there exists an operator $L_0$,  falling under the scope of Theorem~\ref{t1.11}, such that for some $p>1$ one of the properties  {\it $(a)-(f)$} is satisfied for $L_0$. Assume, furthermore,  that 
$$\|A-A_0\|_{L^\infty(\rn)}<\eta$$ for a sufficiently small $\eta>0$  depending on the standard constants and the involved solvability constants of $L_0$, $L_0^*$.  Then  all six assertions {\it $(a)-(f)$} of Theorem~\ref{t1.11} are valid for the operator $L$ as well. \end{corollary}

\begin{corollary}\label{c1.19} Let $L$ be either a divergence form elliptic operator with real (possibly non-symmetric) bounded $t$-independent coefficients, or a perturbation of such an operator, in the sense of Corollary~\ref{c1.16}. Then there exists $p>1$, depending on the dimension and the ellipticity constant of $L$ only, such that statements {\it $(a)-(f)$} of Theorem~\ref{t1.11} are valid for $L$. 

In particular, the Regularity problem is well-posed, for some range of $p>1$, 
for any elliptic operator with real $t$-independent coefficients and for its perturbations.
 \end{corollary}

We note that this result is sharp, in the sense that one cannot specify the range of well-posedness of boundary value problems for real non-symmetric $t$-independent operators which would not depend on the ellipticity parameter of the operator. More precisely, for every $p>1$ there exists an elliptic operator $L$ with real non-symmetric $t$-independent coefficients such that the Regularity problem \eqref{Rp} is not well-posed. A similar statement holds for the Dirichlet problem: given any $p'<\infty$ there exists an elliptic operator $L$ with real non-symmetric $t$-independent coefficients such that the Dirichlet problem \eqref{Dp'} is not well-posed. The counterexample can be found in \cite{KKPT} and \cite{KR} in the context of the Dirichlet and Regularity problem, respectively. 

In the case of real coefficients, the solvability of the Regularity problem for some $L^p$, $p>1$, is equivalent to the solvability in Hardy spaces $H^1$, much as the solvability of the Dirichlet problem in $L^{p'}$ for some $p'<\infty$ is equivalent to the solvability of the Dirichlet problem in $BMO$. For the Dirichlet problem the equivalence was established in \cite{DKP} and for the Regularity problem in \cite{DK}. We refer the reader to \cite{DKP}, \cite{DK} for precise statements. Here we just point out that in the realm of real coefficients the results of Corollary~\ref{c1.19} automatically extend to $H^1$ and $BMO$ spaces for  Regularity and Dirichlet problems, respectively. In fact, the solvability of the Regularity problem for operators with real coefficients can then be further extended to $H^p$, $1-\eps<p\leq 1$. Indeed, one can establish layer potential representations of solutions in $H^1$ (by the same argument as in the proof of Theorem~\ref{t1.11}),  use boundedness of layer potentials in $H^p$ demonstrated in \cite{HMiMo}, and then a Sneiberg-type argument to extrapolate invertibility in $H^1$ to invertibility in $H^p$, $1-\eps<p\leq 1$. This, in turn, can be dualized to get the solvability results in Holder $\dot C^\alpha$ spaces for the Dirichlet problem, with $\alpha>0$ sufficiently close to zero.  Furthermore, having obtained layer potential representations for these solutions, we can use analytic perturbation theory
(cf.  Section \ref{sAppendix} below) to treat complex perturbations of any real coefficient matrix,
thereby extending Corollary \ref{c1.19} to the case of $H^p$ data, $1-\eps<p\leq 1$ (for the regularity problem),
and to $\dot{C}^\alpha, \, \alpha = n(1/p-1)$ and BMO data, for the Dirichlet problem.  We omit the details,  but refer to \cite{HMiMo} where this is done for real symmetric matrices and their perturbations.
Once the layer potential representation has been established, the arguments in \cite{HMiMo} carry over
{\it mutatis mutandi} to the non-symmetric setting.  

Finally, we remark that all the results, in particular, Theorem~\ref{t1.11} and Corollaries~\ref{c1.16}--\ref{c1.19}, automatically extend to Lipschitz domains as defined in \eqref{eq1.3}. This is a consequence of the fact that a ``flattening" change of variables $(x,t)\mapsto (x, t-\psi(x))$, which maps a Lipschitz domain $\{(x,t)\in\reu:\, t>\psi(x)\}$ into $\reu$, preserves the class of $t$-independent elliptic operators. 

Let us now discuss the history of the problem. The study of elliptic boundary problems  \eqref{Dp'}, \eqref{Rp} has started with the results for the Dirichlet problem for the Laplacian on Lipschitz domains \cite{D1}, \cite{D2}. The first breakthrough in the context of the elliptic operators with bounded measurable  coefficients  came in \cite{JK}, where the authors realized how to resolve \eqref{Dp'}, $p'=2$, resting on the so-called Rellich identity. The latter is essentially a result of an integration by parts argument which allows one to compare the tangential and normal derivatives of the solution on the boundary, in the sense that 
\begin{equation}\label{eq1.20}
\|\nabla_{\|}f\|_{L^2(\rn)}\approx \|\partial_{\nu_A} u\|_{L^2(\rn)}, \quad f=u\bigr|_{\rn}.
\end{equation}

The Rellich identity underpinned the development of the elliptic theory for real and symmetric operators, and over the years the problems \eqref{Dp'} and \eqref{Rp} were resolved for the sharp range of $p$ for operators with real symmetric $t$-independent coefficients in \cite{KP} and perturbation results in the spirit of Corollary~\ref{c1.17} were obtained in  \cite{D3},  \cite{FJK},  \cite{FKP}, \cite{KP2}, (see also \cite{AAAHK}, \cite{AA}, \cite{HMM} for later developments in connection with the perturbation questions).

We remark that some ``smoothness"  of the underlying matrix in $t$ is necessary for well-posedness \cite{CFK} and thus starting the investigation with the $t$-independent case is natural in this context.

The argument for the Rellich identity heavily used  the condition of the symmetry of the matrix, and thus, could not be extended neither to real non-symmetric, nor more generally, to the complex case. The only exception to this rule was the resolution of the Kato problem \cite{CMcM, HMc, HLMc, AHLMcT},  in which \eqref{eq1.20} was established in the absence of self-adjointness, 
in the special case that the matrix has block structure, that is, $A=\{A_{jk}\}_{j,k=1}^{n+1}$ with $A_{j,n+1}=A_{n+1,j}=0$, $j=1,...,n$.  The observation that the solution of the Kato problem amounts to  \eqref{eq1.20} is due to C.\,Kenig, see \cite{K}. Moreover, quite recently (simultaneously with the preparation of this manuscript) it was shown in  \cite{AMM},  by a refinement of the proof of the Kato conjecture, 
that the Kato estimate, or more precisely, the $\gtrsim$ side of \eqref{eq1.20}, could be extended to the  ``block triangular" case when only $A_{j,n+1}=0$ without necessarily $A_{n+1,j}=0$ and similarly, the $\lesssim$ side of \eqref{eq1.20} holds when $A_{n+1,j}=0$.

However, the aforementioned ideas could not be directly applied to the general case of a
non-symmetric matrix lacking any additional block structure\footnote{ although the technology of
the Kato problem continues to play  a crucial role in the present paper and in \cite{HKMP}.}. 
Moreover, it was demonstrated in \cite{KKPT} (see also \cite{KR}) that the well-posedness in $L^2$ may fail when matrix has no symmetry and thus, \eqref{eq1.20} is not to be expected. Nonetheless, using a completely different approach, in \cite{KKPT}, \cite{HKMP} the authors have established that for any operator $L$ with real non-symmetric coefficients there is a $p'<\infty$ such that \eqref{Dp'} is solvable. This raised the question of solvability of the Regularity problem. 

In the particular case of the operators with real non-symmetric coefficients in dimension two and their perturbations the Regularity problem was resolved in \cite{KR}, \cite {B}. The present paper resolves this problem in arbitrary dimension. It shows that the solvability of the Dirichlet problem is generally equivalent to that of the Regularity problem and thus, in the context of real non-symmetric matrices there is always a $p$ such that \eqref{Rp} is solvable. The core of our argument is a {\it new Rellich-type inequality}. We demonstrate that, in fact, for any operator $L$ with complex and $t$-independent coefficients  the solvability of the Dirichlet problem \eqref{Dp'}, together with the square function bounds, entails a one-sided Rellich inequality, 
\begin{equation}\label{eq1.21}
 \|\partial_{\nu_A} u\|_{L^p(\rn)}\lesssim \|\nabla_{\|}f\|_{L^p(\rn)}, \quad f=u\bigr|_{\rn}.
\end{equation}
This ultimately paves the way to \eqref{Rp}. Of course, having a reverse inequality as well would be extremely interesting, but at the moment seems quite challenging. 

Finally, we also  point out that the actual question of connections between the Dirichlet and Regularity problem has received considerable attention in the literature, and some partial results were established in \cite{V} (\eqref{Rp} $\Longleftrightarrow$ \eqref{Dp'}, Laplacian on a Lipschitz domain), 
\cite{KP} (\eqref{Rp} $\Longrightarrow$ \eqref{Dp'}, real coefficients), \cite{S}, \cite {KiS} (\eqref{Rp} $\Longleftrightarrow$ \eqref{Dp'}, real symmetric  constant coefficient systems),
 \cite{KR} (\eqref{Dp'}$ \Longrightarrow$ \eqref{Rp},  real coefficients, dimension two),
 where in the non-selfadjoint case, one should understand that these implications hold 
up to taking adjoints. Some related counterexamples were obtained in \cite{M}.
 We also observe that more recently, in the case $p=2$,
 the fact that $(D_2)$ (with square function estimates) $\Longleftrightarrow (R)_2$ (again, up to adjoints), was established explicitly
 in \cite{AA2} (when the domain is the ball, but the proof there carries over to the half-space {\it mutatis mutandi}),  and is at least implicit in the combination of results in \cite[Section 4]{AAMc}
 and \cite[Estimate (5.3)]{AAAHK}, and also in \cite[Section 9]{AA}.
Our main result, Theorem~\ref{t1.11}, generalizes all implications above, at least as far as the $t$-independent matrices are concerned, under the assumption of De Giorgi-Nash-Moser bounds for solutions. The proof of \eqref{eq1.21} builds on Verchota's duality argument \cite{V} , reducing matters to proving $L^{p'}$ estimates for
certain conjugates, and in turn, the estimates for the conjugates will be obtained by an
extension of an argument in \cite{AAAHK} which exploits the solution of the Kato problem.


\section{Preliminaries}\label{sPrelim}

Let $L$ be a divergence-form elliptic operator with $t$-independent  bounded measurable coefficients. Any solution to $Lu=0$ satisfies the  {\it interior Caccioppoli inequality:}

Assume that $Lu=0$ in $\reu$ in the weak sense and $B_{2R}(X)\subset \reu$, $X\in \reu$, $R>0$. Then 
\begin{equation}
 \fiint\limits_{B_R(X)} \left| \nabla u(Y) \right|^2 \,dY \leq \frac{C}{R^2} \fiint\limits_{B_{2R}(X)} \left| u(Y) \right|^2 \,dY,
\label{eq2.1}
\end{equation}
\noindent for some  $C>0$ depending on the dimension and the ellipticity parameter of $L$ only. An analogous statement holds in $\RR^{n+1}_-$.

Recall that we assume, in addition, that solutions of $L$ and $L^*$ satisfy the interior H\"older continuity conditions, that is, the De Giorgi-Nash-Moser estimates \eqref{eq1.6}, \eqref{eq1.7}. 

We proceed to the issues of the non-tangential convergence and uniqueness  of solutions.  Unless explicitly stated otherwise,  we assume throughout the rest of the paper that $L=-\dv(A\nabla)$ is an elliptic operators with complex bounded measurable $t$-independent coefficients and that the solutions to $Lu=0$ and $L^*u=0$ satisfy the De Giorgi-Nash-Moser bounds. Furthermore, throughout the paper we assume that $n\geq 2$, as
 much of this theory in the case $n = 1$ has already been treated in \cite {KR} and \cite{B}.

\begin{lemma}\label{l2.2} Suppose that $u\in W^{1,2}_{loc}(\mathbb{R}^{n+1}_+)$ is a weak solution of $Lu=0$, which satisfies
$\N(\nabla u) \in L^p(\rn)$ for some  $1<p <\infty$.  Then
\begin{itemize}
\item[(i)]  there exists $f\in \dot L^p_1(\rn)$ such that $u \to f$ n.t. a.e., with $$|u(y,t)-f(x)| \lesssim t \,\N(\nabla u)(x),\quad{\mbox{for every $(y,t) \in \Gamma(x)$, $x\in \rn$,}}$$  and $$\|f\|_{\dot L^p_1}\lesssim \|\N(\nabla u)\|_{L^p};$$
\item[(ii)] for the limiting function $f$ from {\rm{(i)}}, one has 
 $$\nabla_\| u(\cdot, t) \longrightarrow \nabla_{\|}f$$
\noindent  as $t \to 0$, in the weak sense in $L^{p}$;
\item[(iii)] there exists $g \in L^p(\rn)$ such that $g=\partial_{\nu_A}u$ in the variational sense, i.e., 
$$\iint\limits_{\reu} A(X) \nabla u(X) \overline{\nabla \Phi(X)}\,dX= \int\limits_{\rn} g(x) \,\overline{\varphi(x)}\,dx,$$ for $\Phi \in \C_0^\infty(\mathbb{R}^{n+1})$ and $\varphi:= \Phi\!\mid_{t=0}$;
\item[(iv)] for the limiting function $g$ from {\rm{(iii)}}, one has 
$$  -e_{n+1} A \nabla u(\cdot, t)  \longrightarrow g$$
\noindent  as $t \to 0$, in the weak sense in $L^{p}$. Here $e_{n+1}=(0,...,0,1)$.
Finally, 
\begin{equation}\label{eq2.3}
\|g\|_{L^p(\rn)}\leq C\, \|\N(\nabla u)\|_{L^p(\rn)}.
\end{equation}
\end{itemize}
\noindent An analogous statement holds in $\RR^{n+1}_-$.
\end{lemma}

The Lemma can be found in \cite{HMiMo} (as stated above), \cite{HMM} (for somewhat more general operators), \cite{AAAHK} (for $p=2$), \cite{Homog} (for real coefficients) and the proof in all cases closely follows an analogous argument in \cite{KP}.

\vskip 0.08in \noindent {\it Remark.} 
We note that the convergence results above entail, in particular, the following. The solvability of \eqref{Rp} as defined in the introduction, that is,  for $C_0^\infty(\rn)$ data, implies existence of a solution to \eqref{Rp}  for any $f\in \dot L^p_1(\rn)$. We refer the reader, e.g., to \cite{Homog}, Theorems~4.6, 5.6, where similar results were established for real symmetric matrices, and the proofs apply to our case without changes.

\begin{lemma}\label{l2.4} For any $2<r< \frac{2(n+1)}{n-1}$ and $1<q<\frac {rn}{n+1}$, $n\geq 1$, and for any $f\in L_1^2(\RR^n)\cap L^{r}(\RR^n)\cap L^q(\RR^n)$ there exists a unique $u\in \widetilde W^{1,2}(\reu)$ such that ${\rm Tr}\,u=f$ and $Lu=0$, in the usual weak sense. Moreover, 
\begin{equation}\label{eq2.5}
\|u\|_{\widetilde W^{1,2}(\reu)}\lesssim \|f\|_{L_1^2(\RR^n)}+\|f\|_{ L^{r}(\RR^n)}+\|f\|_{ L^{q}(\RR^n)}.
\end{equation}
\end{lemma}

\bp The proof is a modification of an analogous arguments \cite{KR}. Here we only mention the main idea. The remaining details are quite easy to fill in,  and if needed, the reader may consult \cite{KR}. 

The basic idea, already mentioned above, is to realize $u$ as $v+w$, where $w$ is the solution to the Laplace's equation with data $f$, that is, the Poisson extension of $f$, and $v$ is the solution to $Lv=G$ (where $G=-Lw$) with zero boundary data, given by the Lax-Milgram Lemma. A direct computation shows that $w\in \widetilde W^{1,2}(\reu)$ satisfies \eqref{eq2.5}. The Lax-Milgram lemma assures existence of the unique solution $v\in \widetilde  W^{1,2}_0(\reu)$ to the problem 
$$ \iint_{\reu} A\nabla v\,{\overline{\nabla \Phi}}\,dxdt= \iint_{\reu} A\nabla w\,{\overline{\nabla \Phi}}\,dxdt, \quad {\mbox{for all}}\quad \Phi\in \widetilde W^{1,2}_0(\reu).  $$
\noindent This requires boundedness and coercivity of the bilinear form with respect to the norm in $\widetilde W^{1,2}(\reu)$. Boundedness is obvious from the definition, and the coercivity follows from the Poincar\'e inequality. We note for the future reference that the Lax-Milgram lemma, in particular, assures that 
\begin{equation}\label{eq2.6}
\|\nabla v\|_{L^2(\reu)}\leq C \|\nabla w\|_{L^2(\reu)}.
\end{equation}

Finally, the desired estimates  \eqref{eq2.5} follow from the combination of aforementioned bounds on $w$ and \eqref{eq2.6}. \ep


Let us now provide somewhat more precise asymptotic estimates on the weak solution $u$ constructed above in the case when $f$ is a nice compactly supported function.

\begin{lemma}\label{l2.7} For any $f\in C_0^\infty(\rn)$ the weak solution $u\in \widetilde W^{1,2}(\reu)$ to the problem $Lu=0$, $u\Bigr|_{\rn}=f$, warranted by Lemma~\ref{l2.4}, satisfies  the following estimates. Let $\Delta_R$ denote the surface ball centered at $O$ such that $\supp f \subset \Delta_R$. Then for all $(x,t)\in\reu$ such that $|(x,t)|>> R$ we have:
\begin{equation}\label{eq2.8}
|u(x,t)|\leq C_f\, t^{-\frac{n+1}2-\eps}, \quad (x,t)\in\reu,
\end{equation}
\noindent uniformly in $x$, provided that $t$ is significantly bigger than the size of the support of $f$. The constant $C_f$ depends on $f$ and the operator $L$, $\eps>0$ depends on the ellipticity parameters of $L$ only.
\end{lemma}

The statement and the proof of this Lemma is the only place in the paper where constants denoted by $C$ are allowed to depend on data $f$. The results, however, will be only used qualitatively to ensure the convergence of some later arising integrals. 

\bp Recall the construction of the solution from the proof of Lemma~\ref{l2.4}. The classical harmonic Poisson extension of $f$, denoted by $w$, clearly satisfies \eqref{eq2.8}. In fact, it decays faster, as $w(X)={\underline{O}}\,(|X|^{-n})$, at infinity.  It remains to estimate $v$. To this end, we decompose $v$ as a sum $\sum_{i=0}^\infty v_i$, where $v_i\in \widetilde  W^{1,2}_0(\reu)$ is the Lax-Milgram solution to  
$$ \iint_{\reu} A\nabla v_i\,{\overline{\nabla \Phi}}\,dxdt= \iint_{\reu} A\nabla w_i\,{\overline{\nabla \Phi}}\,dxdt, \quad {\mbox{for all}}\quad \Phi\in \widetilde W^{1,2}_0(\reu).  $$
Here $w_i=\eta_i w$, with $\eta_i$, $i=0,1,...$,  being the elements of the usual partition of identity associated to dyadic annuli of radius $2^i$ centered at the origin (and $\eta_0$ associated to the unit ball). Then, in particular, for all $i$ sufficiently large depending on the size of the support of $f$ we have
\begin{equation}\label{eq2.9}
\|\nabla v_i\|_{L^p(\reu)}\leq C \|\nabla w_i\|_{L^p(\reu)}\leq C 2^{-i(n+1)/p}, \quad \textstyle{\frac{2(n+1)}{n+3}}-\eps<p\leq 2+\eps,
\end{equation}
\noindent for a suitable $\eps>0$ depending on the ellipticity constants of $L$ only. The first inequality can be seen, e.g., following the reflection procedure to reduce to the case of the entire $\RR^{n+1}$ and then using the Riesz transform bounds from \cite{A}. Note that the elliptic operator arising after reflection is not $t$-independent and does not necessarily satisfy De Giorgi-Nash-Moser bounds (due to boundary effects) but it is a complex coefficient elliptic operator falling under the scope of \cite{A}.

Next, observing that $v_i=0$ on the boundary we employ Poincar\'e inequality to write 
$$ \left(\fiint_{B_k(O)}|v_i|^p\,dX\right)^{1/p}\leq C \,2^k  \left(\fiint_{B_k(O)}|\nabla v_i|^p\,dX\right)^{1/p}\leq C \,2^{k\left(1-\frac{n+1}{p}\right)}\, 2^{-\frac{i(n+1)}{p}}.$$
Finally, by De Giorgi-Nash-Moser estimates, 
\begin{multline*}
|u(x,t)|\lesssim  \left(\fiint_{B_{t/2}(x,t)}|u|^p\,dX\right)^{1/p} \leq  \left(\fiint_{B_{t/2}(x,t)}|w|^p\,dX\right)^{1/p} +\sum_{i=1}^\infty  \left(\fiint_{B_{t/2}(x,t)}|v_i|^p\,dX\right)^{1/p}\\[4pt]
\lesssim t^{-n}+t^{1-\frac{n+1}{p}},
\end{multline*}
\noindent with the implicit constant depending on $f$ and assuming that $t$ is large enough compared to the size of the support of $f$. Given the range of $p$ in \eqref{eq2.9}, this finishes the proof of \eqref{eq2.8}. 
\ep

As discussed above, any  $u\in \widetilde W^{1,2}(\reu)$ has a trace in $\widetilde L^2(\RR^n)$. Moreover, for any $u$ which is a solution to $Lu=0$ in the sense of Lemma~\ref{l2.4}, one can define the conormal derivative $\partial_\nu u=\partial_{\nu_A} u$, in the sense of distributions,  via
\begin{equation}\label{eq2.10}\int\limits_{\rn} \partial_\nu u(x) \,\overline{\varphi(x)}\,dx:=\iint\limits_{\reu} A(X) \nabla u(X) \overline{\nabla \Phi(X)}\,dX,
\end{equation}
\noindent for any $\varphi \in C_0^\infty(\rn)$ and $\Phi \in \C_0^\infty(\mathbb{R}^{n+1})$ such that $\varphi:= \Phi\!\mid_{t=0}$. The details are as follows.

\begin{lemma}\label{l2.11}
Suppose that $u\in \dot{W}^{1,2}(\reu)$, and that $Lu=0$ in $\reu$.  Then $\partial_{\nu_A} u$
exists in $\dot{L}^{2}_{-1/2}(\rn)$;  i.e., there is a $g_u\in \dot{L}^{2}_{-1/2}(\rn)$ such that
for all $H\in \dot{W}^{1,2}(\reu)$,  with trace $tr(H) = h\in \dot{L}^{2}_{1/2}(\rn)$, we have
\begin{equation}\label{eq2.12}
\iint_{\reu}A\nabla u\cdot\overline{\nabla H} \,=\, \langle g_u,\overline{h}\rangle\,,
\end{equation}  
where $\langle\cdot,\cdot\rangle$ denotes here the duality pairing
of $\dot{L}^{2}_{-1/2}(\rn)$ and $ \dot{L}^{2}_{1/2}(\rn)$.  
The analogous statements hold for the adjoint $L^*$, and in the lower half-space.
\end{lemma}

\begin{proof}  We define a bounded linear functional $\Lambda_u$
on $\dot{L}^{2}_{1/2}(\rn)$ as follows.
For $h\in \dot{L}^{2}_{1/2}(\rn)$, set
$$\Lambda_u(h):= \iint_{\reu}A\nabla u\cdot\overline{\nabla H} \,,$$
where $H$ is any $\dot{W}^{1,2}(\reu)$ extension of $h$  (of course, such extensions exist by standard
extension/trace theory).    Note that $\Lambda_u$ is well-defined:  indeed, if $H_1$ and $H_2$
are two different $\dot{W}^{1,2}(\reu)$ extensions
of the same $h$, then $H_1-H_2\in \dot{W}_0^{1,2}(\reu)$, whence
it follows that 
$$\iint_{\reu}A\nabla u\cdot\overline{\nabla (H_1-H_2)} \,=\,0 \,,$$
since $Lu=0$ in the weak sense.  Moreover, it is obvious that $\Lambda_u$ is linear.
To see that the functional is bounded, we simply choose an extension $H$ (e.g., the
harmonic extension), for which
$$\|H\|_{\dot{W}^{1,2}(\reu)} \leq\, C_0\,\|h\|_{\dot{L}^{2}_{1/2}(\rn)}\,,$$
for some purely dimensional constant $C_0$.  We then have
$$|\Lambda_u(h)|\,\leq\, \|A\|_\infty\, \|u\|_{\dot{W}^{1,2}(\reu)}\,\|H\|_{\dot{W}^{1,2}(\reu)}
\,\leq \, C_0 \|A\|_\infty\, \|u\|_{\dot{W}^{1,2}(\reu)}\,\|h\|_{\dot{L}^{2}_{1/2}(\rn)}\,,$$
i.e., $\|\Lambda_u\|\lesssim \|u\|_{\dot{W}^{1,2}(\reu)}$.   The conclusion of Lemma \ref{l2.11}
now follows by the Riesz Representation Theorem.
\end{proof}

Lemma~\ref{l2.11} allows us to justify the definition of the conormal derivative by \eqref{eq2.10} for all functions $u\in \widetilde{W}^{1,2}(\reu)$, by first identifying such a function with the corresponding equivalence class in $\dot{W}^{1,2}(\reu)$, then applying Lemma~\ref{l2.11} and then deducing that for any particular representative of this equivalence class, in particular, for $u$ itself, we have \eqref{eq2.10}. Such a definition gives identical result for $u$ in the same equivalence class, but, naturally,  the conormal derivative would not distinguish functions that differ by a constant.

 Note that if, in addition, $\N(\nabla u)\in L^p(\RR^n)$, for some $1<p<\infty$, then $\partial_\nu u\in L^p(\rn)$ by Lemma~\ref{l2.2}.

\section{Boundary estimate: a version of the Rellich-type inequality} 

\begin{theorem}\label{t3.1}
Let $L$ be an elliptic operator with $t$-independent coefficients such that the solutions to $Lu=0$ and $L^*u=0$ in $\repm$ satisfy the De Giorgi-Nash-Moser estimates.  Let $u$ be a solution to the Dirichlet problem $Lu=0$ in $\reu$, $u\Bigl|_{\partial\reu}=f$, for some $f\in C_0^\infty(\RR^n)$, in the sense of Lemma~\ref{l2.4}. Suppose that for some $1<p'<\infty$ the Dirichlet problem \eqref{Dp'} for the operator $L^*$ is solvable with the square function bounds, that is,  for every $C_0^\infty$ boundary data the  corresponding weak solution satisfies both  \eqref{eq1.4} and  the square function estimate \eqref{eq1.12}. Then the variational derivative of $u$ defined by \eqref{eq2.10} can be identified with an $L^p$ function, and
\begin{equation}\label{eq3.2}
\|\partial_{\nu_A} u\|_{L^p(\RR^n)} \leq C\, \|\nabla_\| f\|_{L^p(\RR^n)}, 
\end{equation}
\noindent where $\frac{1}{p}+\frac{1}{p'}=1.$  The constant $C$ depends on the standard constants and  on the solvability constants of $L^*$ involved in \eqref{eq1.4} and \eqref{eq2.10}. 
\end{theorem}

\bp We aim to show that for $u$, a solution to $Lu=0$ in $\reu$, the normal derivative on the boundary is controlled by the tangential derivatives for some $1<p<\infty$.
 To this end, take $1<p'<\infty$ such that \eqref{Dp'} for the operator $L^*$ is solvable with the square function bounds,  
 consider any $g\in C_0^\infty$ with $\|g\|_{L^{p'}}\leq 1$, and denote by $w$ the weak solution to \eqref{Dp'} for the operator $L^*$ with boundary data $g$. Then 
\begin{equation}\label{eq3.3}
\int_\rn \partial_{\nu_A} u\, {\overline{g}}\, dx=\iint\limits_{\reu}  \nabla F(X) \,{\overline{A^*(X) \nabla w(X)}}\,dX= \int_\rn f\, {\overline{\partial_{\nu_{A^*}}w}}\, dx,
\end{equation}
\noindent for any $F\in C_0^\infty(\ree)$ such that $F\Bigr|_{\rn}=f$. This follows simply from the definition of the weak conormal derivative in \eqref{eq2.10}. Indeed, by definition
$$ \int_\rn f\, {\overline{\partial_{\nu_{A^*}}w}}\, dx =\iint\limits_{\reu}  \nabla F(X) \,{\overline{A^*(X) \nabla w(X)}}\,dX=\iint\limits_{\reu}  \nabla u(X) \,{\overline{A^*(X) \nabla w(X)}}\,dX,$$
\noindent where for the second equality we use the fact that $u-F\in  \widetilde W^{1,2}_0(\reu)$ and $w$ is a solution in the sense of \eqref{eq1.2} (evidently, $C_0^\infty$ functions are dense in $\widetilde W^{1,2}_0(\reu))$. A similar argument applies to show that 
$$\int_\rn \partial_{\nu_A} u\, \overline g\, dx=\iint\limits_{\reu} A(X) \nabla u(X) \, {\overline{\nabla w(X)}}\,dX.$$

We shall take an extension of $f$ in the form $F(x,t):=f(x)\eta_{r,R}(t)$, $R>>r$, where $\eta_{r,R}(t)\in C_0^\infty((-(R+r), R+r))$, $\eta_{r,R}(t)=1$ for $t\in (-R,R)$, and $|\eta'|\leq 1/r$. Here $r$ is chosen so that $\Delta_r$ contains $\supp (f)$. Then the right-hand side of \eqref{eq3.3} is equal to 
\begin{multline}\label{eq3.4}
\iint\limits_{\Delta_r\times (0,R)}  \nabla F(X) \,{\overline{A^*(X) \nabla w(X)}}\,dX+ \iint\limits_{\Delta_r\times (R, r+R)}  \nabla F(X) \,{\overline{A^*(X) \nabla w(X)}}\,dX\\[4pt]
= \sum_{j=1}^n \iint\limits_{\Delta_r\times (0,R)}  \partial_j f(x) \,{\overline{e_j \,A^*(X) \nabla w(X)}}\,dX + \iint\limits_{\Delta_r\times (R, r+R)}  \nabla F(X) \,{\overline{A^*(X) \nabla w(X)}}\,dX
\\[4pt]
= \sum_{j=1}^n \int\limits_{\Delta_r}  \partial_j f(x) \,\left(\int_0^R {\overline{e_j \,A^*(x) \nabla w(x,t)}}\,dt\right)\,dx + \iint\limits_{\Delta_r\times (R, r+R)}  \nabla F(X) \,{\overline{A^*(X) \nabla w(X)}}\,dX,
\end{multline}

\noindent where $e_j=(0,...,0,1,0,...,0)$ is the $j$-th basis vector of $\ree$, and we used the fact that by construction $F=f$ in $\Delta_r\times (0,R)$ and hence, in this range, it is independent of $t$.

We remark that, intuitively, one can think of the functions in parentheses above as an analogue of harmonic conjugates. That is, 
given a solution $w$ to $Lw=0$ in $\reu$,   a system of $L$-harmonic conjugates could be defined as follows: 
\begin{equation}\label{eq3.5}
\widetilde w_j(x,t):=-\int_{t}^\infty e_j\,A(x)\,\nabla w(x,s)\,ds, \qquad (x,t)\in \reu, \quad j=1,...,n+1,
\end{equation}

\noindent (see, e.g., \cite{FS} and \cite{KR} for analogous constructions in the case of harmonic functions and variable-coefficient operators in dimension 2, respectively).  Thus, \eqref{eq3.3}--\eqref{eq3.4} and forthcoming calculations are actually manipulations with harmonic conjugates in disguise. However, because of the weak nature of the available definition of solution and conormal derivative, we have to carefully keep track of the error terms. 

Going further, departing from the right-hand side of \eqref{eq3.4}, we can write 
\begin{multline}\label{eq3.6}
\int_\rn \partial_{\nu_A} u\, {\overline{g}}\, dx\\[4pt]
=\sum_{i,j=1}^n \int\limits_{\Delta_r}  \partial_j f(x) \,\left(\int_0^R {\overline{A_{ji}^*(x) \partial_i w(x,t)}}\,dt\right)\,dx\\[4pt] 
-\sum_{j=1}^n \int\limits_{\Delta_r}  \partial_j f(x) \, {\overline{A_{j,n+1}^*(x) g(x)}}\,dx +
\sum_{j=1}^n \int\limits_{\Delta_r}  \partial_j f(x) \, {\overline{A_{j,n+1}^*(x)  w(x,R)}}\,dx 
\\[4pt]+ \iint\limits_{\Delta_r\times (R, r+R)}  \nabla F(X) \,{\overline{A^*(X) \nabla w(X)}}\,dX=: I_R+II+E_{1,R}+E_{2,R}.
\end{multline}

First of all, we claim that the terms $E_{1,R}$ and $E_{2,R}$ both vanish as $R\to\infty$. Indeed, 
$$ E_{2,R} \leq C_{f,r} \left( \,\iint\limits_{2\Delta_r\times (R-r, 2r+R)} |w|^2\,dX\right)^{1/2} \leq C_{f,r} \,R^{1-\frac{n+1}{2}},$$
using the Cauchy-Schwarz and Caccioppoli inequalities for the first bound and Lemma~\ref{l2.7} for the second one. Analogously, using  Lemma~\ref{l2.7}, we see that $E_{1,R} \leq C_{f,r} R^{1-\frac{n+1}{2}}$ as well.

It remains to analyze $I_R$ and $II$. The integral in $II$ directly gives the desired bound by $\|\nabla_\|f\|_{L^p}$. The estimate on $I_R$ is trickier. Using, as before, the decay of $w$ assured by Lemma~\ref{l2.7}, we see that it is enough to bound
\begin{equation}\label{eq3.7} I:= \int_{\rn} \nabla_\| f(x)\cdot {\overline{A^*_\|(x)\nabla_\| v(x)}}\,dx, \end{equation}
where formally
\begin{equation}\label{eq3.8}
v(x,t):=\int_t^\infty w(x,s)\,ds, \qquad v(x):=v(x,0),
\end{equation}
as the error vanishes as $R\to\infty$. Here and in the sequel, $A_\|$ denotes an $n\times n$ block of the matrix $A$, that is, $\{A_{jk}\}_{j,k=1}^n$, and $L_\| =-\dv_\| A_\|\nabla_\|$, interpreted, as usually, in the weak sense.  

Let us discuss the definition of $v$. First of all,  $v(x,t)$ itself is well-defined for any $t>0$ as an absolutely convergent integral (using Lemma~\ref{l2.7}), and $\nabla v(\cdot, t)$ belongs to $L^2_{loc}(\rn)$ for any $t>0$ (again, using Lemma~\ref{l2.7} and Proposition~2.1 in \cite{AAAHK}). Further, $\nabla_\| v$ on the boundary is well-defined as an  $L^2_{loc}(\rn)$ function (using the fact that $w\in \widetilde W^{1,2}(\reu)$ and aforementioned considerations), and $\nabla_\|v(\cdot, t)$ converges to $\nabla_\|v(\cdot)$ in the sense of distributions and in $L^2_{loc}(\rn)$.
This clarifies the sense of \eqref{eq3.7}--\eqref{eq3.8}. We note for the future reference that $\partial_tv=w$ belongs to $\widetilde W^{1,2}(\reu)$ and satisfies Lemma~\ref{l2.7}, as well as \eqref{eq2.1}.

We claim that 
\begin{equation}\label{eq3.9}
\left|\int_{\rn} \nabla_\| f(x)\cdot {\overline{A^*_\|(x)\nabla_\| v(x)}}\,dx\right|\\[4pt]
\lesssim \|\nabla_\| f\|_{L^p}\left(\|\A (t\nabla \partial_t v)\|_{L^{p'}}+\|N_*(\partial_tv)\|_{L^{p'}}\right), 
\end{equation}

\noindent where   $\A$, as before,  stands for the square function \eqref{eq1.13}.
Estimate \eqref{eq3.9} is one of the core components of our approach to the regularity problem and we will concentrate on its proof in the next section. It is interesting to point out that the particular quadratic form appearing here is important:  even though \eqref{eq3.9} holds, one cannot expect to
deduce that $\nabla_\|
v \in L^{p'}$, except in the range $2-\eps<p<2+\eps$
(for us, $1<p<2+\eps$ in any case, so
only the lower bound here is a further restriction).   Indeed, specializing to the block case,
an $L^{p'}$ bound for $\nabla_\| v$ for $p'>2+\eps$ (i.e, $p<2-\eps$),
would contradict the counter-example of Kenig  (see \cite[pp 119-120]{AT}).
This has to do with the failure of the Hodge decomposition
for $L_\|:= -\dv_\| A_\| \nabla_\|$, outside of the stated range of $p$.
Indeed, to test the $L^{p'}$ norm of
$A^*_\| \nabla_\|v$ (which by ellipticity of $A^*_\|$ is equivalent to the
$L^{p'}$ norm of $\nabla_\|v$), would  require
that we test against an arbitrary vector $\vec{h} \in L^p(\rn, \mathbb{C}^n)$.   Observe that
the form $\int \vec{h} \cdot \overline{A^*_\|\nabla_\| v}$ is equivalent to the 
form on the left-hand side of \eqref{eq3.9}, only if 
we have a Hodge
decomposition $\vec{h} = \nabla_\| f + \vec{g}$, where $\vec{g}$ satisfies
$\div_\| A_\|\, \vec{g} =0$, and $f\in \dot{L}^p_1$.  But such a Hodge decomposition holds only
in the range $2-\eps<p<2+\eps$.   

For now, let us finish the proof of the Theorem assuming \eqref{eq3.9}. 

Recall that $\partial_t v=w$ and by definition $w$ is the solution to \eqref{Dp'} for $L^*$ satisfying both the non-tangential maximal function and the square function estimates. All in all, then \eqref{eq3.6}--\eqref{eq3.9} guarantee that for any $g\in L^{p'}$,
\begin{equation}\label{eq3.10}
\int_\rn \partial_{\nu_A} u\, {\overline{g}}\, dx \lesssim \|\nabla_\| f\|_{L^p} \|g\|_{L^{p'}}.
\end{equation}

\noindent  Hence, 
\begin{equation}\label{eq3.11}
\|\partial_{\nu_A} u\|_{L^p(\RR^n)} \lesssim \|\nabla_\| f\|_{L^p(\RR^n)}, 
\end{equation}

\noindent with $f=u|_\rn$. This finished the proof of Theorem ~\ref{t3.1}, modulo \eqref{eq3.9}.  \ep

\section{Proof of the main  estimate (\ref{eq3.9}) 
}\label{s2}

In this section, we establish the ``main estimate" \eqref{eq3.9} (re-stated as Theorem \ref{t4.10}
below)
thereby completing the proof of Theorem \ref{t3.1}. 
We shall adapt the proof of \cite[estimate (5.3)]{AAAHK}, which is 
essentially the case $p=2$ of \eqref{eq3.9}, and which exploits the
solution of the Kato problem.
In our case, we require $L^p$ versions of the Kato estimate (cf.  \eqref{eq4.2} below).

Let us start by recalling a few results regarding the square roots of elliptic operators and the corresponding square functions that will be used throughout the proof. Retain the definitions of $N_*$, $\N$ and $\A$ from Sections 1 and 3, and let
\begin{eqnarray*} 
\mathcal{A}_q^{\alpha}(F)(x) &=& \left(\iint\limits_{|x-z|<\alpha t} \left|F(z,t)\right|^q \frac{dzdt}{t^{n+1}} \right)^{1/q}, \\
{\C}_{q}(F)(x) &=& \sup_{\Delta \ni x} \left(\frac{1}{|\Delta|} \iint\limits_{\widehat{\Delta}} |F(y,s)|^q \frac{dyds}{s} \right)^{1/q}, 
\end{eqnarray*} 
where, as before,  $\widehat{\Delta}=\{(x,t): \text{dist}(x,\Delta^c)\geq t\}$ and $\Delta(x,t)= \left\{y \in \rn: \left| x-y \right| < t\right\}$.  The aperture index $\alpha$ will usually be omitted unless it plays an explicit role in the proof. Also, as per \eqref{eq1.13}, $\A=\A_2$, and $\C=\C_2$.

For a Lebesgue measurable set $E$, we let ${\bf M}(E)$ denote the
collection of  measurable functions on $E$.  For $0<p,q<\infty$ we define the following tent spaces:
\begin{eqnarray*} 
T^p_q(\reu) &=& \left\{F\in {\bf M}(\mathbb{R}^{n+1}_+): {\A}_q(F) \in L^p(\rn)\right\},\\ 
T_{\infty}^p(\reu) &=& \left\{F\in {\bf M}(\mathbb{R}^{n+1}_+):N_*(F) \in  L^p(\rn)\right\},\\
\widetilde{T}_{\infty}^p(\reu) &=& \left\{F\in {\bf M}(\mathbb{R}^{n+1}_+):\N(F) \in  L^p(\rn)\right\},\\ 
T_{q}^{\infty}(\reu) &=&\left \{F\in {\bf M}(\mathbb{R}^{n+1}_+): {\C}_{q}(F) \in L^{\infty}(\rn)\right \},\\ 
{\bf T}_{q}^{\infty}(\reu) &=& \left\{F\in {\bf M}(\mathbb{R}^{n+1}_+): {\C}_{q}({\bf s}(F)) 
\in L^{\infty}(\rn)\right \},
\end{eqnarray*}
where in the last definition, for $F\in{\bf M}(\mathbb{R}^{n+1}_+)$, we set
$${\bf s}(F)(x,t):= \sup_{(y,s)\in W(x,t)} \left|F(y,s)\right|$$
(the notation ``sup" is interpreted as the essential supremum) and $W(x,t) = \Delta(x,t) \times (t/2, 3t/2)$. The spaces $T^p_q(\reu)$, $0<p,q\leq \infty$ were first introduced by Coifman, Meyer and Stein in \cite{CMS}. The spaces $\widetilde{T}_{\infty}^p(\reu)$ and ${\bf T}_{q}^{\infty}(\reu)$ started appearing in the literature more recently, naturally arising for elliptic PDEs with non-smooth coefficients.

As usual, we say that a family of operators $\{T_t\}_{t>0}$
satisfies $L^p-L^q$ off-diagonal estimates, $1\leq p\leq q\leq \infty$, if for arbitrary
closed sets $E,F\subset \RR^n$
\begin{equation}\label{eq4.1}
\|T_tf\|_{L^q(F)}\leq Ct^{\left(\frac nq-\frac
np\right)}\,e^{-\frac{{\rm dist}\,(E,F)^2}{ct}}\,\|f\|_{L^p(E)},
\end{equation}

\noindent for every $t>0$ and every $f\in L^p(\RR^n)$ supported in
$E$. We remark that whenever $L$ and $L^*$ both satisfy the De Giorgi-Nash-Moser property, 
the heat semigroup  $\p_t:= e^{-t^2 L_\|}$, $t>0$, satisfies pointwise Gaussian upper bounds, and hence, $L^p-L^q$ off-diagonal estimates for all $1\leq p,q\leq \infty$. Indeed, if the solutions of $L$ have the De Giorgi-Nash-Moser bounds, then so do the solutions to $L_\|$. To see that, let $\Delta =\Delta(x,r)$ be an $n$-dimensional ball, and
let $B=B((x, r),r/2)$ be the corresponding $(n+1)$-dimensional ball.
Let $u=u(x)$ solve $L_\| u=0$ in $2\Delta$.  Set
$U(x,t):= u(x)$, so that $U$ is $t$-independent.  Since $U$
and also the coefficients of $L$ are $t$-independent, we have that
$LU(x,t) = L_\| u(x) = 0$ for all  $(x,t)\in 2B$. Since $U$ satisfies the De Giorgi-Nash-Moser property in $B$, we conclude that  $u$ satisfies the De Giorgi-Nash-Moser property in $\Delta$. Furthermore, if the solutions to both $L_\|$ and $L^*_\|$ have the De Giorgi-Nash-Moser bounds, then the 
 heat kernel, that is, the kernel of the semigroup  $\p_t:= e^{-t^2 L_\|}$, $t>0$, satisfies pointwise Gaussian upper bounds, and enjoys  Nash type local H\"older continuity, by \cite{AT} (Theorem 10, p. 34, {\it{loc. cit.}}), as desired.

The latter, in turn, imply that the 
square root estimate, 
\begin{equation}\label{eq4.2} \left\|\sqrt{L_\|} f\right\|_{L^p}\leq C \left\|\nabla f\right\|_{L^p},  
\end{equation}
holds for all $1<p<\infty$, $f\in C_0^\infty$, (the case $p=2$ corresponds to the Kato problem solved in \cite{AHLMcT}, and the generalization to other values of $p$ (given the $L^2$ case), in the presence of Gaussian bounds, can be found in \cite{AT}). The Gaussian bounds also imply that
\begin{eqnarray}\label{eq4.3}  &&\left\|\A(\theta_t f)\right\|_{L^p}\leq C \|f\|_{L^p},\\[4pt]
&&\qquad \mbox{with $\theta_t= t\sqrt{L_\|}\,\p_t$, $\theta_t= t^2\nabla_\|\sqrt{L_\|}\,\p_t$, $\theta_t= t^3 L_\|^{3/2}\,\p_t$, $\theta_t= t^4\nabla_\| L_\|^{3/2}\,\p_t$, }\nonumber
\end{eqnarray}
holds for all $1<p<\infty$, $f\in L^p$. In the required generality these square function estimates do not seem to be explicitly stated anywhere, but they are all essentially well known.
 Indeed, one may verify the case $p=2$ by a standard ``quasi-orthogonality" argument;   for $p> 2$,
one may follow the well known argument of  \cite{FS} to get a Carleson measure estimate when 
$f\in L^\infty$, and then use tent space interpolation to obtain all $p\in [2,\infty)$;  for $p<2$,
one may first prove the Hardy space bound $\left\|\A(\theta_t f)\right\|_{L^1}\leq C \|f\|_{H^1}$,
by a standard argument using the atomic decomposition of $H^1$, and the local H\"older continuity 
of the heat kernel, and then interpolate to get the full range of $p$.
We omit the details.

Going further, we record the following result essentially following from the Poincar\'e inequality. It was proved for $p=2$ in \cite{AAAHK} (Lemma 3.5, {\it{loc.~cit.}}).

\begin{lemma}\label{l4.4} Assume that  a family of operators $\{R_t\}_{t>0}$ satisfies $L^2-L^2$ off-diagonal estimates and that $R_t 1=0$ (in the sense of $L^2_{loc}(\RR^n)$). Then 
\begin{equation}\label{eq4.5}
\|\A(R_t F)\|_{L^p}\leq C \|\A(t\nabla_\| F)\|_{L^p}, 
\end{equation} 
for every $F$ with $t\nabla_\| F\in T^p_2$, i.e., such that the right-hand side of \eqref{eq4.5} is finite, and for every $1<p<\infty$. 
\end{lemma}

In fact, a weaker off-diagonal decay rate of $R_t$ than the exponential estimates above  would suffice, but for all relevant choices of $R_t$'s in the present paper the exponential decay will be valid. 

\bp Let us denote $\Delta_{x,t}=\{y\in \RR^n:\,|x-y|<t\}$, $S_j(\Delta_{x,t})=\Delta_{x, 2^{j+1}t}\setminus \Delta_{x, 2^{j}t}$, $j\in\NN$, and $F_{x,t}:=\fint_{\Delta_{x,t}} F(y,t)\,dy$. Then 
\begin{multline}\label{eq4.6}
\|\A^1(R_t F)\|_{L^p}=\left(\int_{\RR^n}\left(\iint_{|x-y|< t}|R_tF(y,t)|^2 \frac{dy dt}{t^{n+1}} \right)^{p/2}dx\right)^{1/p}\\[4pt]
\leq \left(\int_{\RR^n}\left(\iint_{|x-y|< t}|R_t([F(\cdot,t)-F_{x,2t}]\,\chi_{\Delta_{x,2t}})(y)|^2 \frac{dy dt}{t^{n+1}} \right)^{p/2}dx\right)^{1/p}\\[4pt]
+ \sum_{j=1}^{\infty} \left(\int_{\RR^n}\left(\iint_{|x-y|< t}|R_t([F(\cdot,t)-F_{x,2t}]\,\chi_{S_j(\Delta_{x,t})})(y)|^2 \frac{dy dt}{t^{n+1}} \right)^{p/2}dx\right)^{1/p}\\[4pt]
=I+II.
\end{multline}

Using the uniform in $t$ boundedness of $R_t$ in $L^2(\RR^n)$ (following from the $L^2-L^2$ off-diagonal estimates) and then the Poincar\'e inequality, we deduce that 
\begin{multline}\label{eq4.7} I\lesssim  \left(\int_{\RR^n}\left(\iint_{|x-y|< 2t}|F(y,t)-F_{x,2t}|^2 \frac{dy dt}{t^{n+1}} \right)^{p/2}dx\right)^{1/p} \\[4pt]
\lesssim  \left(\int_{\RR^n}\left(\iint_{|x-y|< 2t}|t\nabla_\|F(y,t)|^2 \frac{dy dt}{t^{n+1}} \right)^{p/2}dx\right)^{1/p} =  \|\A^2(t\nabla_\| F)\|_{L^p}.
\end{multline}

\noindent On the other hand, by $L^2-L^2$ off-diagonal estimates,
\begin{multline}\label{eq4.8} II\lesssim  \sum_{j=1}^{\infty} \left(\int_{\RR^n}\left(\int_0^\infty e^{-(2^jt)^2/(ct^2)}\int_{S_j(\Delta_{x,t})}|F(y,t)-F_{x,2t}|^2\, dy \,\frac{ dt}{t^{n+1}} \right)^{p/2}dx\right)^{1/p} \\[4pt]
\lesssim  \sum_{j=1}^{\infty}\sum_{k=1}^j \left(\int_{\RR^n}\left(\int_0^\infty e^{-\,C\,4^j} \left(\frac {2^j}{2^k}\right)^n\int_{\Delta_{x,2^{k+1}t}}|F(y,t)-F_{x,2^{k+1}t}|^2\, dy \,\frac{ dt}{t^{n+1}} \right)^{p/2}dx\right)^{1/p}, \\[4pt]
\end{multline}

\noindent where we used the representation 
$$F(y,t)-F_{x,2t}=F(y,t)-F_{x,2^{j+1}t}+ F_{x,2^{j+1}t}-F_{x,2^{j}t}+...+F_{x,2^{2}t}-F_{x,2t}$$
for the second inequality. Using now the Poincar\'e inequality, we have 
\begin{multline}\label{eq4.9} II
\lesssim  \sum_{j=1}^{\infty}\sum_{k=1}^j \left(\int_{\RR^n}\left(\int_0^\infty e^{-\,C\,4^j}2^{jn} 2^{-kn}\, 2^{2k}\int_{\Delta_{x,2^{k+1}t}}|t\nabla_\| F(y,t)|^2\, dy \,\frac{ dt}{t^{n+1}} \right)^{p/2}dx\right)^{1/p} \\[4pt]\lesssim  \sum_{k=1}^\infty 2^{-Mk} \left(\int_{\RR^n}\left(\int_0^\infty \int_{\Delta_{x,2^{k+1}t}}|t\nabla_\| F(y,t)|^2\, dy \,\frac{ dt}{t^{n+1}} \right)^{p/2}dx\right)^{1/p}, \\[4pt]
\end{multline}

\noindent where $M>n/2$ can be arbitrarily large constant. However, \eqref{eq4.9} simply says that 
$$ II
\lesssim \sum_{k=1}^\infty 2^{-Mk} \left\|\A^{2^{k+1}}(t\nabla_\| F)\right\|_{L^p},$$
which in turn implies that 
$$ I+II
\lesssim \sum_{k=0}^\infty 2^{-Mk+C(n,p) k} \|\A^{1}(t\nabla_\| F)\|_{L^p}\lesssim \|\A^{1}(t\nabla_\| F)\|_{L^p},$$
since $M$ can always be taken large enough. We remark that the sharp constant $C(n,p)$ appearing in the change-of-aperture square function estimates was obtained in \cite{AuscherCR}, although for the purposes of the present argument we only need to know that the dependence on the aperture is polynomial, and this was already established in \cite{CMS}.  
\ep

At this point we are ready to turn to the proof of estimate \eqref{eq3.9}. 
\begin{theorem}\label{t4.10} Assume that $L$ is an elliptic operator with $t$-independent coefficients, and that $L$ and $L^*$ satisfy the De Giorgi-Nash-Moser bounds.  Let $w$ be the weak solution to the Dirichlet problem for $L^*$ with some $C_0^\infty$ data, guaranteed by Lemma~\ref{l2.4}, and $v$ be its antiderivative, defined by \eqref{eq3.8}. 
Then 
\begin{multline}\label{eq4.11}
\left|\int_{\rn} \nabla_\| f(x)\cdot {\overline{A^*_\|(x)\nabla_\| v(x,0)}}\,dx\right|\\[4pt]
\leq C \,\|\nabla_\| f\|_{L^p}\left(\|\A(t\,\nabla (\partial_t v))\|_{L^{p'}}+\|N_*(\partial_tv)\|_{L^{p'}}\right)
\end{multline}

\noindent for every  $f\in C_0^\infty(\RR^n)$ and $1<p<\infty$. Here  $C>0$ depends on  the  standard constants  only. 
\end{theorem}

Clearly, \eqref{eq4.11} is only of interest when the right-hand side of \eqref{eq4.11} is finite. It is the case, e.g.,  when the Dirichlet problem is solvable in $L^{p'}$, and the solution satisfies the square function estimates. However, this information is not needed to establish \eqref{eq4.11}.

\noindent{\it Remark.} For future reference, we point out that the argument will establish a more general result. To be specific, for $L$ and $f$ as in the statement of Theorem~\ref{t4.10}  we shall demonstrate that
\begin{multline}\label{eq4.12}
\left|\int_{\rn} \nabla_\| f(x)\cdot {\overline{A^*_\|(x)\nabla_\| v(x,0)}}\,dx\right|\\[4pt]
\leq C \,\|\nabla_\| f\|_{L^p}\left(\|\A(t\,\nabla (\partial_t v))\|_{L^{p'}}+\|N_*(\partial_tv)\|_{L^{p'}}\right)
\\[4pt]+\left|\iint_{\reu} \partial_t^2\p_t   f(x)\,{\overline{L^* v(x,t)}}\,tdtdx\right|, 
\end{multline}

\noindent for every $v:\reu\to\RR$ with reasonable decay properties sufficient to justify convergence of involved integrals. Here $\p_t:= e^{-t^2 L_\|}$, $t>0$, is, as before, the heat semigroup associated to the operator $L_\|$. Clearly, when $v$ is a solution, as in the statement of Theorem~\ref{t4.10}, the last integral on the right hand side of \eqref{eq4.12} is equal to zero and \eqref{eq4.12} reduces to \eqref{eq4.11}.

\noindent {\it{Proof of Theorem~\ref{t4.10}}}. As discussed in the paragraph above \eqref{eq3.9}, the left-hand side of \eqref{eq4.11} is an absolutely convergent integral since $\nabla_\|v(\cdot, 0)\in L^2_{loc}(\rn)$ and $f\in C_0^\infty(\rn)$. It will be convenient though to work with a particular approximation. 

Let us recall that $f_\eps:=\p_\eps f$ converges to $f$ as $\eps\to 0$ weakly in $\dot L^p_1(\rn)$ for all $1<p<2+\eps$. One way to see this is to invoke the bound 
$$ \|\N(\nabla \p_t f)\|_{L^p} \leq C  \|\nabla_\|f\|_{L^p}, \quad 1<p<2+\eps, $$
(it was proved for the Poisson semigroup in \cite{M}, but the same proof, or, actually, its simplified version, applies to the case of the heat semigroup as well). We only remark that the {\it full} range of $1<p<2+\eps$ in the bounds for the heat semigroup $\p_\eps$ is achievable due to the De Giorgi-Nash-Moser estimates on $L$ which yield the  Gaussian bounds on the heat semigroup of $L_\|$ (see the discussion preceding \eqref{eq4.2}).  Having this at hand, we can use the general fact that the estimates on the non-tangential maximal function of  the gradient imply weak convergence to the boundary data in $\dot L^p_1(\rn)$ (see the statements (i) and (ii) of Lemma~2.2 in the current manuscript and recall that they do not use the assumption that $u$ is a solution). In addition, we know that the off-diagonal decay of the heat semigroup assures that $t^k\partial_t^k \p_tf$, as well as $t^{k+1}\nabla_\|\partial_t^k \p_tf$, $k=0,1,2,...,$ decay exponentially away from the support of $f\in C_0^\infty(\rn)$ in the sense of Gaffney off-diagonal estimates, that is, \eqref{eq4.1} holds (with a restriction $q<2+\eps$ in the case of estimates on $t^{k+1}\nabla_\|\partial_t^k \p_tf$).

It follows, in particular, that $\int_{\rn} \nabla_\| \p_\eps f(x)\cdot {\overline{A^*_\|(x)\nabla_\| v(x,0)}}\,dx$ converges to the left-hand side of \eqref{eq4.11} as $\eps \to 0$: we can apply off-diagonal decay of the gradient of the heat semigroup far away from the support of $f$ and the aforementioned weak convergence in the remaining portion of the integral, as $\nabla_\|v \in L^2_{loc}(\rn)$. 
 
Now, with $f_\eps=\p_\eps f$,  
\begin{multline}\label{eq4.13}
\int_{\rn} \nabla_\| f_\eps (x)\cdot {\overline{A^*_\|(x)\nabla_\| v(x,0)}}\,dx\\[4pt]
=-\iint_{\reu} \frac{\partial}{\partial t}\left(\nabla_\| \p_t   f_\eps (x)\cdot {\overline{A^*_\|(x)\nabla_\| v(x,t)}}\right)\,dtdx\\[4pt]
=-\iint_{\reu} \nabla_\| \partial_t\p_t   f_\eps(x)\cdot {\overline{A^*_\|(x)\nabla_\| v(x,t)}}\,dtdx\,-\iint_{\reu} \nabla_\| \p_t   f_\eps (x)\cdot {\overline{A^*_\|(x)\nabla_\| \partial_tv(x,t)}}\,dtdx\\[4pt]
=\iint_{\reu} \nabla_\| \p_t   f_\eps (x)\cdot {\overline{A^*_\|(x)\nabla_\| \partial_t^2v(x,t)}}\,tdtdx\\[4pt]
\qquad +2 \iint_{\reu} \nabla_\| \partial_t\p_t   f_\eps (x)\cdot {\overline{A^*_\|(x)\nabla_\| \partial_tv(x,t)}}\,tdtdx
\\[4pt]
\,\qquad +\iint_{\reu} \nabla_\| \partial_t^2\p_t   f_\eps (x)\cdot {\overline{A^*_\|(x)\nabla_\| v(x,t)}}\,tdtdx =: I+II+III.
\end{multline}

\noindent  For any fixed $\eps>0$ the convergence of integrals and integration by parts above and through the argument is justified by our assumptions on $v$  (in particular, the properties of $v$ and $\partial_tv$ outlined in the paragraph above \eqref{eq3.9}), and off-diagonal estimates for $\p_t$ and its derivatives.

Then 
\begin{multline}\label{eq4.14}
|I|=\left| \iint_{\reu} L_\| \p_t   f_\eps (x) {\overline{\partial_t^2v(x,t)}}\,tdtdx\right|\\[4pt]
\lesssim \left\|\A(t\sqrt{L_\|} \p_t (\sqrt {L_\|} f_\eps ))\right\|_{L^p} \, \left\|\A(t\partial_t^2v)\right\|_{L^{p'}} \\[4pt]
\lesssim \left\|\A(t\sqrt{L_\|} \p_t (\p_\eps  \sqrt {L_\|} f))\right\|_{L^p} \, \left\|\A(t\partial_t^2v)\right\|_{L^{p'}}\\[4pt]
\lesssim \|\p_\eps  \sqrt {L_\|} f)\|_{L^p} \, \left\|\A(t\partial_t^2v)\right\|_{L^{p'}} \lesssim  \|\nabla_\| f\|_{L^p}\, \left\|\A(t\partial_t^2v)\right\|_{L^{p'}},
\end{multline}

\noindent using boundedness of the square function based on $t\sqrt{L_\|} \p_t $ \eqref{eq4.3}, uniform in $\eps$ bounds on $\p_\eps$ in $L^p$, and then the Kato estimate \eqref{eq4.2}, all available for $1<p<\infty$. Largely by the same argument, 
\begin{multline}\label{eq4.15}
|II|
\lesssim \left\|\A(t^2\nabla_\|\sqrt{L_\|} \p_t (\sqrt {L_\|} f_\eps ))\right\|_{L^p} \, \left\|\A(t\nabla_\|\partial_tv)\right\|_{L^{p'}} \\[4pt]\lesssim  \|\nabla_\| f\|_{L^p}\, \left\|\A(t\nabla_\|\partial_tv)\right\|_{L^{p'}}.
\end{multline} 

The estimate on $III$ is more delicate. We write 
\begin{multline}\label{eq4.16}
III= \iint_{\reu} \partial_t^2\p_t   f_\eps (x)\,{\overline{L^* v(x,t)}}\,tdtdx \\[4pt]
+ \sum_{j=1}^{n} \iint_{\reu} \partial_t^2\p_t   f_\eps (x)\,{\overline{\partial_j A_{j,n+1}^* \partial_tv(x,t)}}\,tdtdx
\\[4pt]
+ \sum_{j=1}^{n+1} \iint_{\reu} \partial_t^2\p_t   f_\eps (x)\, {\overline{A_{n+1,j}^* \partial_j\partial_tv(x,t)}}\,tdtdx=:III_1+III_2+III_3.
\end{multline} 

\noindent The term $III_1$ is left alone for the moment.  It is zero when $v$ is a solution  (note that $t\partial_t^2\p_t   f_\eps=t\partial_t^2\p_t   \p_\eps f \in \ring W^{1,2}(\reu)$ for any fixed $\eps>0$) and otherwise, it shows up explicitly on the right hand side of \eqref{eq4.12}. 
The term $III_3$ can be handled just as $I$ and $II$ above and gives the bound 
\begin{equation}\label{eq4.17}
III_3\lesssim  \|\nabla_\| f\|_{L^p}\, \left\|\A(t\nabla\partial_tv)\right\|_{L^{p'}},
\end{equation}

\noindent once again using \eqref{eq4.3} and \eqref{eq4.2}.

As for $III_2$, 
\begin{multline}\label{eq4.18}
III_2=C \sum_{j=1}^{n} \iint_{\reu} \partial_j\partial_t^2\p_{t}   f_\eps(x)\,{\overline{A_{j,n+1}^* (I-P_t)\partial_tv(x,t)}}\,tdtdx\\[4pt] \quad +C \sum_{j=1}^{n} \iint_{\reu} \partial_t^2\p_{t}   f_\eps(x)\,{\overline{\left(\partial_j A_{j,n+1}^* P_t\partial_tv(x,t)\right)}}\,tdtdx =: III_{2,1}+III_{2,2},
\end{multline} 

\noindent where $P_t$ is a nice approximation of identity, e.g., the heat semigroup of the Laplacian. Since $R_t:=I-P_t$ kills constants, and, by virtue of standard heat kernel bounds, satisfies off-diagonal decay estimates,  one can use the Poincar\'e inequality to produce the gradient and get an estimate on $III_{2,1}$ akin to that for $I+II$. Indeed, according to Lemma~\ref{l4.4}, square function estimates \eqref{eq4.3} and Kato estimate \eqref{eq4.2}, 
$$ III_{2,1} \lesssim \|\p_\eps \sqrt{L_\|} f\|_{L^p} \, \left\|\A((I-P_t)\partial_tv)\right\|_{L^{p'}}\lesssim  \|\nabla_\| f\|_{L^p}\, \left\|\A(t\nabla\partial_tv)\right\|_{L^{p'}}.$$

 Concerning $III_{2,2}$, we further write
\begin{multline}\label{eq4.19}
III_{2,2}\\[4pt]=C \sum_{j=1}^{n} \iint_{\reu} \left(4t^2L_\|^2e^{-t^2L_\|}-2L_\|e^{-t^2L_\|}\right)   f_\eps(x)\,{\overline{\left(\partial_j A_{j,n+1}^* P_t\partial_tv(x,t)\right)}}\,tdtdx\\[4pt]= \sum_{j=1}^{n} \iint_{\reu} \left(C_1t^2L_\|^2e^{-\frac 34 t^2L_\|}+C_2L_\|e^{-\frac 34 t^2L_\|}\right)   f_\eps(x)\,{\overline{\p^*_{t/2}\left(\partial_j A_{j,n+1}^* P_t\partial_tv(x,t)\right)}}\,tdtdx\\[4pt]=  \sum_{j=1}^{n} \iint_{\reu} \left(C_1t^2L_\|^2e^{-\frac 34 t^2L_\|}+C_2L_\|e^{-\frac 34 t^2L_\|}\right)   f_\eps (x)\times\,\\[4pt]\qquad\qquad\qquad\qquad \times {\overline{\left(\p^*_{t/2}\partial_j A_{j,n+1}^* P_t\partial_tv(x,t)-(\p^*_{t/2}\partial_j A_{j,n+1}^*) (P_t\partial_tv(x,t))\right)}}\,tdtdx\\[4pt]
+  \sum_{j=1}^{n} \iint_{\reu} \left(C_1t^2L_\|^2e^{-\frac 34 t^2L_\|}+C_2L_\|e^{-\frac 34 t^2L_\|}\right)  f_\eps (x)\,\times\,\\[4pt]\qquad \times{\overline{(\p^*_{t/2}\partial_j A_{j,n+1}^*) \,(P_t\partial_tv(x,t))}}\,tdtdx=III'_{2,2}+III''_{2,2}, 
\end{multline} 

\noindent where $(\p^*_{t/2}\partial_j A_{j,n+1}^*)\,(P_t\partial_tv(x,t))$ is interpreted as a product of two functions, while $\p^*_{t/2}\partial_j A_{j,n+1}^* P_t\partial_tv(x,t)$, as before, is a result of an operator $\p^*_{t/2}\partial_j A_{j,n+1}^* P_t$ acting on the function $\partial_tv$.

Much as in the analysis of $III_{2,1}$, we note that the operator 
$$ R_t:=t\left(\p^*_{t/2}\partial_j A_{j,n+1}^* P_t-(\p^*_{t/2}\partial_j A_{j,n+1}^*) P_t\right)$$
\noindent kills constants and satisfies $L^2-L^2$ off-diagonal estimates (see, e.g., \cite{AHLMcT}), so that by Lemma~\ref{l4.4} one can bound $III'_{2,2}$ similarly to $III_{2,1}$. Finally, 
\begin{multline}\label{eq4.20}
III''_{2,2}\lesssim \sum_{j=1}^{n} \left\|\C(t\,\p^*_{t/2}\partial_j A_{j,n+1}^*)\right\|_{L^\infty}\,\times\,\\[4pt]\qquad \times\left\|\A\left(\left(C_1t^3L^2e^{-\frac 34 t^2L}+C_2tLe^{-\frac 34 t^2L}\right)   f_\eps\right)\right\|_{L^p}\,\left\|N_*(P_t\partial_tv)\right\|_{L^{p'}}.
\end{multline} 

\noindent The latter estimate follows from the duality of tent spaces $T^p_2$ and $T^{p'}_2$ and the fact that $T_2^\infty \cdot T^{p'}_{\infty} \hookrightarrow T_2^{p'}$ (see \cite{CMS} for the case $p'>2$ and \cite{HMM} or \cite{CV}, p. 313 for all $0<p'<\infty$).

 For each $j$, the estimate on the Carleson measure of 
$\mathcal{C}(t\,\p^*_{t/2}\partial_j A_{j,n+1}^*)$ is  controlled, since the 
Kato problem for $L_\|^*$  is equivalent to the statement that
$$\|\mathcal{C}(t\p^*_{t/2}\nabla_\|\cdot \vec{b})\|_{L^\infty(\mathbb{R}^n)} \leq
C \|\vec{b}\|_{L^\infty(\mathbb{R}^n)},$$ 
for every $\vec{b}\in L^\infty(\mathbb{R}^n, \mathbb{C}^n)$
\cite{AHLMcT}. Hence, 
\begin{equation}\label{eq4.21}
III''_{2,2}\lesssim \left\|\nabla_\|f\right\|_{L^p}\,\left\|N_*(\partial_tv)\right\|_{L^{p'}}.
\end{equation} 

\noindent Indeed, recall that $P_t$ is a ``nice" approximation of identity, e.g., the heat semigroup of the Laplacian and hence, one can drop it inside the non-tangential maximal function using the fact that $N_*(P_t (F(\cdot, t))\leq M(N_* F)$ with $M$ denoting the Hardy-Littlewood maximal function. \ep

This finishes the proof of Theorem~\ref{t3.1}.

\section{Boundedness of layer potentials and estimates for  solutions in $\reu$.}

Let us start by recalling a few known estimates on the layer potentials and related operators that will be used throughout this section. 

\begin{proposition}\label{p5.1} Suppose that $L$ is $t$-independent, that $L$
and its adjoint
$L^*$ satisfy the De Giorgi-Nash-Moser bound \eqref{eq1.6}. Then the single layer potential satisfies the following estimates  in $L^2$: 
\begin{equation} \int_{-\infty}^{\infty}\int_{\mathbb{R}^{n}}\left|t\partial_t\nabla S^L_{t}f(x)\right|^{2}\frac{dx dt}{|t|}\,\leq\, C\,\Vert
f\Vert_{L^2(\rn)}^{2},\label{eq5.2}\end{equation}
\begin{equation}\sup_{t\neq0}\left( \|\left(S^{L^*}_t\nabla\right)f\|_{L^2(\rn)}\,
+\, \|\nabla S^{L}_t f\|_{L^2(\rn)}
\right)
\lesssim \, \|f\|_{L^2(\rn)},\label{eq5.3}\end{equation}
\begin{equation}\label{eq5.4}
\|\N(\nabla S^L_{\pm t}f)\|_{L^2(\rn)}\lesssim\|f\|_{L^2(\rn)}\,.
\end{equation}  
\noindent Moreover, there exists $\eps >0$ such that for all $1<p<2+\eps$ we have 
\begin{equation}\label{eq5.5}
\|\N(\nabla S^L_{\pm t}f)\|_{L^p(\rn)}\lesssim\|f\|_{L^p(\rn)}\,,
\end{equation}
and for the dual exponent $p'$,
\begin{eqnarray}\label{eq5.16}
&&\|N_*((S^{L^*}_{\pm t} \nabla)f)\|_{L^{p'}(\rn)}\lesssim\|f\|_{L^{p'}(\rn)},\\
\label{eq5.16-bis} &&\|\A(t\nabla S_{\pm t}^{L^*}\nabla f)\|_{L^{p'}}\lesssim
\|f\|_{L^{p'}(\rn)}\,. 
\end{eqnarray} 
\noindent
Analogous bounds hold for $L^*$. The implicit constants in all inequalities and $\eps>0$ depend upon   the standard constants  only. 
\end{proposition}

Here and below 
\begin{align}\label{eq5.6}
\left( S_t D_j\right)f(x)&:= \int_{\rn}\frac{\partial}{\partial y_j} E(x,t;y,0)\,f(y)\,dy\,,\quad 
1\leq j\leq n,\\[4pt]\label{eq5.7}
\left( S_t D_{n+1}\right)f(x)&:= \int_{\rn}\frac{\partial}{\partial s} E(x,t;y,s)\big|_{s=0}\,f(y)\,dy\,,
\end{align}
and we set 
\begin{equation}\label{eq5.8}\left( S_t\nabla\right):=\Big(\left( S_t D_1\right),\left( S_t D_2\right),...,
\left( S_t D_{n+1}\right)\Big)\,,\,\,\,{\rm and} \,\, \left( S_t\nabla\right)\cdot \vec{f}:=\sum_{j=1}^{n+1}
\left( S_t D_j\right)f_j\,,
\end{equation}
where $\vec{f}$ takes values in $\mathbb{C}^{n+1}$. Similarly, 
\begin{equation}\label{eq5.9}\left( S_t\nabla_\|\right):=\Big(\left( S_t D_1\right),\left( S_t D_2\right),...,
\left( S_t D_{n}\right)\Big)\,,\,\,\,{\rm and} \,\, \left( S_t\nabla_\|\right)\cdot \vec{f}:=\sum_{j=1}^{n}
\left( S_t D_j\right)f_j\,,
\end{equation}
where $\vec{f}$ takes values in $\mathbb{C}^{n}$.

Note that for $t$-independent operators, we have by translation invariance in $t$ that
$\left( S_t D_{n+1}\right)=\,-\,\partial_t S_t$.  

 \begin{remark}\label{r5.11}
 We remark that, considering only the tangential gradient
$\nabla_\|$ in \eqref{eq5.16}, the latter may be re-formulated as
\begin{equation}\label{eq5.15}
 S^{L^*}_{\pm t}: L^{p'}_{-1} \to T^{p'}_{\infty}(\reu).
 \end{equation}
 We further note that \eqref{eq5.5}, \eqref{eq5.16}, $t$-independence of $L$, and the De Giorgi-Nash-Moser bounds immediately imply that 
\begin{equation}
\label{eq5.18}  \partial_t S_t^L: L^p(\RR^n)\to T^p_\infty(\reu), \quad 1<p<\infty,
\end{equation}
under the assumptions of Proposition~\ref{p5.1}.
\end{remark}

\noindent {\it Proof of Proposition~\ref{p5.1}.} The square function 
bound \eqref{eq5.2} was proved in \cite{R}  (for an alternative proof, 
see  \cite{GH}).  
The fact that \eqref{eq5.2} 
implies \eqref{eq5.3} and \eqref{eq5.4} is basically a combination of results in \cite{AAAHK} and \cite{AA}. See Proposition~1.19 in \cite{HMM} for a detailed discussion and references. The fact that \eqref{eq5.3} for $L$ and $L^*$ implies \eqref{eq5.5}  and  \eqref{eq5.16}
has been proved 
 in \cite[Theorem 1.1 and estimate (4.45)]{HMiMo}. Finally, \eqref{eq5.16-bis} can be found in \cite{HMM}.\ep

\begin{proposition}\label{p5.10} Suppose that $L$ is $t$-independent, that $L$
and its adjoint
$L^*$ satisfy the De Giorgi-Nash-Moser bound \eqref{eq1.6}. Then 
\begin{eqnarray}
\label{eq5.11} 
&& \nabla L^{-1}\dv: T^p_2 \to \widetilde{T}^p_\infty, \quad 1<p< 2+\eps, \\[4pt]
\label{eq5.12}
&& \nabla L^{-1} \textstyle{\frac 1t}: T^p_1\to \widetilde{T}^p_\infty, \quad 1<p< 2+\eps,
\end{eqnarray}
\noindent for some $\eps>0$, depending on the standard constants only. Here the operator  $ L^{-1} \textstyle{\frac 1t}$ is to be interpreted via
\begin{equation}\label{eq5.13} \left(L^{-1}\textstyle{\frac 1t} \Psi\right) (y,s):=\iint_{\reu} E(y,s; x,t) \,\Psi(x,t)\,\frac{dxdt}{t}, \quad (y,s)\in \reu. \end{equation}
\end{proposition}

\bp The Proposition was proved in \cite {HMM}.\ep





\begin{proposition}\label{p5.19} Suppose that $L$ is $t$-independent, that $L$
and its adjoint
$L^*$ satisfy the De Giorgi-Nash-Moser bound \eqref{eq1.6}. Then there exists $\eps >0$ such that for all $1<p<2+\eps$,  we have
\begin{equation}\label{eq5.20}
\nabla D^L_{\pm t}: \dot L^{p}_{1} \to T^{p}_{\infty}(\reu).
 \end{equation}
\noindent In particular,
\begin{equation}\label{eq5.21}
\|\widetilde N(\nabla D^L_{\pm t} f)\|_{L^{p}(\rn)}\leq C \|\nabla_\| f\|_{L^{p}(\rn)}.
\end{equation}
\noindent
 Here $C>0$ and $\eps>0$  depend upon the standard constants only. Analogous bounds hold for $L$.
\end{proposition}

\bp By the usual density considerations, it is enough to demonstrate \eqref{eq5.21} for $f\in C_0^\infty(\rn)$.  
To start, let us concentrate on the estimate for $\partial_t D_t^L$, that is,
 \begin{equation}\label{eq5.23}
\|\widetilde N(\partial_t D_t^L f)\|_{L^p(\rn)} \lesssim \|\nabla_\| f\|_{L^p(\RR^n)}.
\end{equation}

To this end, we note that $L^*_{y,s} E^*(y,0;x,t)=0$ for all $(x,t)\in\reu$, since $(y,0)$ is on the boundary.  Hence,  one can formally write 
\begin{multline}\label{eq5.24}
\partial_t\int_\rn {\overline{\partial_{\nu_{A^*},y} \,E^*(y,0;x,t)}}\, f(y)\,dy=\int_\rn {\overline{\partial_s e_{n+1}A^*(y)\,\nabla_{y,s}\,E^*(y,s;x,t)\Big|_{s=0}}}\, f(y)\,dy\\[4pt]
= -\int_\rn \sum_{j=1}^n {\overline{\partial_j e_{j}\,A^*(y)\,\nabla_{y,s}\,E^*(y,s;x,t)\Big|_{s=0}}}\, f(y)\,dy
\\[4pt]
= \int_\rn {\overline{L^*_{\|,y}\,E^*(y,0;x,t)}}\, f(y)\,dy - \int_\rn \sum_{j=1}^n {\overline{\partial_j \,A^*_{j,n+1}(y)\,\partial_s\,E^*(y,s; x,t)\Big|_{s=0}}}\, f(y)\,dy
\\[4pt]
= \int_\rn {\overline{A^*_\|(y)\nabla_y\,E^*(y,0;x,t)}}\, \nabla_y f(y)\,dy \,\\[4pt] +\,\int_\rn \sum_{j=1}^n {\overline{A^*_{j,n+1}(y)\,\partial_s\,E^*(y,s;x,t)\Big|_{s=0}}}\, \partial_j f(y)\,dy. 
\end{multline}

\noindent Of course, the formal computations in \eqref{eq5.24} should be interpreted in the weak sense, using, in particular, the weak definition of the normal derivative \eqref{eq2.10} (see the discussion in proof of Theorem~\ref{t5.32} for more details). 

The desired estimate on the second term on the right-hand side of \eqref{eq5.24} follows from \eqref{eq5.18}. Passing to the first term, recall that the tent space $\widetilde T^p_\infty$ can be realized as a space of the linear functionals on $\widetilde T^{p'}_1$, where the latter is defined as a collection of $F\in {\bf M}(\reu)$ such that  

$$\|F\|_{\widetilde T^{p'}_1}:=\left\|\C_1(W_2F)\right\|_{L^{p'}}<\infty, $$
\noindent  where
$$W_2 F(x,t)=\left(\fint\!\!\!\!\!\!\fint\limits_{W(x,  t)}|F(y,s)|^{2}dyds\right)^{\frac{1}{2}}, \quad (x,t)\in\reu. $$

\noindent Indeed, it was proved in \cite{HR}, Theorem 3.2, that  $\widetilde T^p_\infty=\left(\widetilde T^{p'}_1\right)^*$, $1<p<\infty$. Pairing the first term on the right-hand side of \eqref{eq5.24} with $\Phi$, one obtains
\begin{multline}\label{eq5.25}
\iint_{\reu }\int_\rn {\overline{A^*_\|(y)\nabla_y\,E^*(y,0;x,t)}}\, \nabla_y f(y)\,dy \,{\overline{\Phi(x,t)}}\,\frac{dxdt}{t}\\[4pt]= \int_\rn {\overline{A^*_\|(y)\nabla_y\,T \Phi(y,0)}}\, \nabla_y f(y)\,dy, 
\end{multline}
where $T=(L^*)^{-1}\,\frac 1t$, that is, as before,
\begin{equation}\label{eq5.26} T \Phi (y,s)=T_s\Phi(y):=\iint_{\reu} E^*(y,s;x,t) \,\Phi(x,t)\,\frac{dxdt}{t}. \quad (y,s)\in \reu.\end{equation}

\noindent The  goal is to show that 
\begin{equation}\label{eq5.27}\left|  \int_\rn {\overline{A^*_\|(y)\nabla_y\,T \Phi(y,0)}}\, \nabla_y f(y)\,dy\right|\lesssim \|\Phi\|_{\widetilde T^{p'}_1}\,\|\nabla_\|f\|_{L^p}, \end{equation}
\noindent for any $\Phi\in \widetilde T^{p'}_1$.

It is sufficient to verify \eqref{eq5.27} for $\Phi$ smooth and compactly supported in $\reu$, as such functions are dense in $ \widetilde T^{p'}_1$. At this point, recall the estimate \eqref{eq4.11} or, more precisely, \eqref{eq4.12}. One can carefully track the proof of \eqref{eq4.11}--\eqref{eq4.12} to see that all the computations are justified for $v:=T\Phi$ with $\Phi$ smooth and compactly supported in $\reu$. However, $v=T\Phi$ is not a solution in $\reu$ and hence, the term $III_1$ (that is, the last integral on the right-hand side of \eqref{eq4.12}) will not be annulated.  Instead, we have $L^*v(x,t)=L^*T\Phi(x,t)=\Phi(x,t)/t$ for $(x,t)\in\reu$.  

All in all, it is enough to bound the right-hand side of \eqref{eq4.12} with $v=T\Phi$, $\Phi\in \widetilde T_1^{p'}(\reu)$.
Thus,  one has to prove: 
$$ \|\A(t\,\nabla \partial_t T_t\Phi)\|_{L^{p'}(\rn)} \lesssim \|\Phi\|_{\widetilde T_1^{p'}(\reu)}, \qquad \|N_*(\partial_t T_t\Phi)\|_{L^{p'}(\rn)} \lesssim \|\Phi\|_{\widetilde T_1^{p'}(\reu)},$$

\noindent or, equivalently, in the language of tent spaces, 
\begin{eqnarray}\label{eq5.28} 
&&\left\|t\nabla \partial_t T_t\Phi\right\|_{T_2^{p'}(\reu)} \lesssim \|\Phi\|_{\widetilde T_1^{p'}(\reu)}, \qquad 
\\[4pt] \label{eq5.29} &&\left\|\partial_t T_t\Phi\right\|_{T_\infty^{p'}(\reu)}  \lesssim \|\Phi\|_{\widetilde T_1^{p'}(\reu)},
\end{eqnarray}

\noindent and to bound the last term on the right-hand side of \eqref{eq4.12} with ${\overline{f}}$ in place of $f$. However,  
\begin{equation}\label{eq5.30}
\iint_{\reu} \partial_t^2\p_t  {\overline{f}}(x)\,L^* T\Phi(x,t)\,tdtdx=\iint_{\reu} \partial_t^2\p_t   {\overline{f}}(x)\,\Phi(x,t)\,dtdx,
\end{equation}

\noindent so that using once again duality relationship for tent spaces the desired bound on \eqref{eq5.30} reduces to 
\begin{equation}\label{eq5.31} 
\|\widetilde N(t\partial_t^2\p_t   {\overline{f}})\|_{L^{p}(\rn)} \lesssim \|{\overline{\nabla_\| f}}\|_{L^p(\RR^n)}=\|\nabla_\| f\|_{L^p(\RR^n)}.
\end{equation} 
\noindent The estimate \eqref{eq5.31} follows from the bound $\widetilde N(t\partial_t^2\p_t   f)\lesssim M(\nabla_\|f)$, while the latter can be proved essentially by the same  argument as that in \eqref{eq4.6}--\eqref{eq4.9}, using  the off-diagonal decay estimates on $t^2\partial_t^2\p_t$ and the Poincar\'e inequality. 

It remains to discuss \eqref{eq5.28} and \eqref{eq5.29}. A direct computation shows that the adjoint of the operator $\partial_tT$, under the usual tent space pairing $\langle \Phi,\Psi\rangle=\iint_{\reu}\Phi\,\overline{\Psi}\, \frac{dxdt}{t}$, is $-\partial_s L^{-1}\,\frac 1t$ (since by $t$-independence one can swap the derivatives in $t$ and $s$ on the fundamental solution). Since $T^p_\infty=\left(T^{p'}_1\right)^*$, $1<p<\infty$, (see, e.g., \cite{HR}, Theorem 3.2, or the duality argument in \cite{CMS}),  \eqref{eq5.29} follows from \eqref{eq5.12} in the range $1<p<2+\eps$. 

Finally, the adjoint of the operator $t\nabla \partial_tT$, under the tent space pairing $\langle \Phi,\Psi\rangle=\iint_{\reu}\Phi\,\Psi\, \frac{dxdt}{t}$, is $-\partial_s L^{-1}\dv$, and  $T^p_2=\left(T^{p'}_2\right)^*$, $1<p<\infty$, (see \cite{CMS}), so that 
\eqref{eq5.28} follows from \eqref{eq5.11}, once again, in the range $1<p<2+\eps$, as desired.

This finishes the proof of \eqref{eq5.23}. 

We claim that the desired bound \eqref{eq5.21} follows from \eqref{eq5.23}, or, to be precise, from the estimate 
 \begin{equation}\label{eq5.23-bis}
\|N_*(\partial_t D_t^Lf)\|_{L^p(\rn)} \leq C \|\nabla_\| f\|_{L^p(\RR^n)}.
\end{equation}
Since $\partial_t D_t^Lf$ is a solution, \eqref{eq5.23-bis} is equivalent to \eqref{eq5.23} by the De Giorgi-Nash-Moser estimates. 

Indeed, using the argument  in \cite{KP}, p. 494, we see that  
the left-hand side of \eqref{eq5.21} is controlled by the sum of $\|M(\nabla_\| D_t^L \Big|_{t=0}f)\|_{L^p(\RR^n)}$ and $\|M(N_*(\partial_t D_t^Lf))\|_{L^p(\RR^n)}$, where $M$ 
 is the Hardy-Littlewood maximal function on $\RR^n$. The first of these two terms is bounded by $\|\nabla_\| D_t^L \Big|_{t=0}f\|_{L^p}$. The second one is bounded by the left-hand side of \eqref{eq5.23-bis}.
 
 Thus, it remains to show that 
\begin{equation}\label{eq5.31.1}
\nabla_\| D_t^L \Big|_{t=0}: \dot L^p_1(\rn)\to L^p(\rn), \quad 1<p<2+\eps.
\end{equation} 
Let $f, g \in C^\infty_0(\rn)$, (where $g$ is vector-valued) and set  $w(x,t):= S_{-t}^{L^*} \div_\| g(x)$, $x\in\rn$, so that, by \eqref{eq5.16},  
\begin{equation}\label{eq5.31.2}
 \|w(\cdot, 0)\|_{L^{p'}(\rn)} \lesssim \|g\|_{L^{p'}(\rn)}.
\end{equation} 

\noindent Note that $w$ is a solution (for $L^*$) in the lower half-space, and choose an appropriate solution $v$
in the lower half-space such that  $\partial_t v = w$.   Dualizing, and using the equation,
we then have
\begin{multline}
\int_{\rn}\nabla_\| D_t\Big|_{t=0} f(x)\,  \overline{g(x)}\,dx  =   \int_{\rn} f(x) \,\overline{\partial_{\nu_{A*}} w(x)} \,dx=   \int_{\rn} f(x) \,\overline{\partial_t \partial_{\nu_{A*}} v(x)} \,dx  \\
  =  -\sum_{j=1}^{n}\int_{\rn} \partial_j f(x) \,\overline{A^*_{j,n+1} \partial_t v(x)} \,dx- \int_{\rn }\nabla_\| f\,\overline{  A_\|^* \nabla_\| v}\,dx =: I + II.
 \end{multline}  

\noindent For term $I$, we just use \eqref{eq5.31.2}, and for term $II$, we 
apply estimate \eqref{eq3.9} and then \eqref{eq5.16} and \eqref{eq5.16-bis} (one can go over the argument of \eqref{eq3.9} and justify integration by parts and convergence of involved integrals even though the data of $w$ on the boundary is  not $C_0^\infty$).
\ep

\begin{theorem} \label{t5.32} Let $L$ be an elliptic operator with $t$-independent coefficients such that the solutions to $Lu=0$ and $L^*u=0$ in $\repm$ satisfy the De Giorgi-Nash-Moser estimates. Let $u$ be the solution to the Dirichlet problem $Lu=0$ in $\reu$, $u\Bigl|_{\partial\reu}=f$, for some $f\in C_0^\infty(\RR^n)$, in the sense of Lemma~\ref{l2.4}. 
Then  for $1<p<2+\eps$
\begin{equation}\label{eq5.33}
\|\widetilde N(\nabla u)\|_{L^p(\rn)} \leq C \left(\|\nabla_\| f\|_{L^p(\RR^n)}+  \|\partial_{\nu_{A}} u\|_{L^p(\RR^n)}\right),\end{equation}
\noindent  with $C$ and $\eps$ depending on the standard constants only.
\end{theorem}

\bp 
By Green's formula, the solution to $Lu=0$ in $\reu$, $u\Bigr|_{\partial\reu}=f$, can be written as 
\begin{equation}\label{eq5.34}u(x,t) =-\int_\rn {\overline{\partial_{\nu_{A^*},y} \,E^*(y,0;x,t)}}\, f(y)\,dy+\int_\rn E(x,t; y,0)\, \partial_{\nu_{A}} u(y)\,dy,\end{equation}
\noindent or, in the language of layer potentials, 
\begin{equation}\label{eq5.35} u=-D_t^L (f)+S_t^L(\partial_{\nu_{A}} u). \end{equation}

\noindent With \eqref{eq5.35} at hand, \eqref{eq5.33} follows directly from \eqref{eq5.21} and \eqref{eq5.5}. 

It remains to justify the use of the Green's formula in the present context.

To this end, fix $X_0:=(x_0,t_0)\in\reu$, and let $\eps>0$, with $\eps< t_0/8$.   We define a nice approximate
identity on $\ree$ in the standard way as follows.  Let 
$\Phi\in C^\infty_0(\ree)$,
with $\Phi\geq 0$, $\supp(\Phi)\subset B(0,1)$, and $\iint_{\ree}\Phi=1$, and set
$\Phi_\eps(X):= \eps^{-n-1} \Phi(X/\eps)$.  For $(y,s)\in\ree$, we set
$$H_\eps(y,s) := \Big(\Phi_\eps * E^*(y,s,\cdot,\cdot)\Big)(x_0,t_0)\,.$$
We observe that by definition of the fundamental solution $E^*$,
$$H_\eps = \left(L^*\right)^{-1} \Phi^{X_0}_\eps\,,$$
where $\Phi^{X_0}_\eps (X):= \Phi_\eps(X_0-X)\in C^\infty_0(B(X_0, \eps))$.
In particular, then, we have that $H_\eps\in \dot{W}^{1,2}(\ree)$.

Let $F$ be a $C^\infty_0(\ree)$ extension of $f = u(\cdot,0)$, and set
\begin{equation*}
\tilde{u}(x,t):=\left\{
\begin{array}{l}
u(x,t)\,, 
\,\,t\geq 0
\\[4pt]
F(x,t)\,,
\,\,t\leq 0\,.
\end{array}
\right.
\end{equation*}
Since $\eps\ll t_0$,  and since $L^* H_\eps = \Phi^{X_0}_\eps$  in the weak sense in $\ree$
(and in particular, $L^*  H_\eps = 0$ in the lower half-space),
we have
\begin{multline}\label{eq5.36}\Phi_\eps*u(X_0) =
 \iint_{\ree}\tilde{u}(X)\,\Phi^{X_0}_\eps(X)\,  dX 
= \iint_{\ree} \nabla \tilde{u} \cdot \overline{A^*\nabla H_\eps}\\[4pt]
=  \iint_{\reu} A\nabla u \cdot \overline{\nabla H_\eps} \,+ \, 
\iint_{\mathbb{R}^{n+1}_-} \nabla F \cdot \overline{A^*\nabla H_\eps}
\,=\, \langle \partial_{\nu_A} u\,,\, \overline{h_\eps}\rangle\, -\, 
\langle \overline{\partial_{\nu_{A^*}}H_\eps}\,,\, f\rangle\,,
\end{multline}
where $h_\eps:= H_\eps(y,0)$, and where in the last step we have used 
Lemma~\ref{l2.11}.
By definition of $H_\eps$, \eqref{eq5.35}  follows, letting $\eps\to 0$, once we establish the convergence in $\eps$ for both terms on the right-hand side of \eqref{eq5.36}. 

The convergence of the first term is quite easy, as $h_\eps \to E^*$ in $\dot{L}^2_{1/2} (\rn)$.

Let us discuss the second one. In this regard, one has to show  that 
for $\eps\ll t_0$, the variational co-normal derivative $-\partial_{\nu_{A^*}}H_\eps$, initially defined in the sense of
Lemma~\ref{l2.11}, belongs to $L^2(\rn)$, and satisfies
\begin{equation}\label{eq5.37}-\partial_{\nu_{A^*}}H_\eps(y,0) =
e_{n+1}\cdot A^* \Big(\Phi_\eps*\nabla_{y,s} E^*(y,s,\cdot,\cdot)\big|_{s=0}\Big)(X_0).
\end{equation}

That the right hand side of \eqref{eq5.37} belongs to $L^2(\rn)$ follows directly from 
\cite[Lemma 2.8]{AAAHK}.  Thus we need only verify the identity \eqref{eq5.37}.
By Lemma~\ref{l2.11}, applied in the lower half-space, it is enough to verify that for any
$F\in C^\infty_0(\ree)$, with $F(\cdot,0):= f$, we have
\begin{equation}\label{eq5.38}\iint_{\RR^{n+1}_-}\nabla F\cdot\overline{A^*\nabla H_\eps} \,=\, \int_{\rn} f (y)\,
\overline{\vec{N}\cdot A^* \Big(\Phi_\eps*\nabla_{y,s} E^*(y,s,\cdot,\cdot)\big|_{s=0}\Big)(X_0)}\,dy\,.
\end{equation}
To this end, let $P_\eta$ be a nice approximate identity in $\rn$, applied in the $y$ variable.
By the divergence theorem 
\begin{multline*}\int_{\rn} f (y)\,
\overline{\vec{N}\cdot P_\eta\left(A^* \Big(\Phi_\eps*\nabla_{y,s} E^*(y,s,\cdot,\cdot)\big|_{s=0}\Big)(X_0)
\right)}\,dy \\[4pt]=\, \iint_{\RR^{n+1}_-}\dv \left(F(y,s)\,\overline{
 P_\eta\left(A^* \Big(\Phi_\eps*\nabla_{y,s} E^*(y,s,\cdot,\cdot)\Big)(X_0)
\right)}\right)\, dyds\\[4pt]=\, \iint_{\RR^{n+1}_-}\nabla F(y,s)\cdot\overline{
 P_\eta\left(A^* \nabla H_\eps(y,s)
\right)}\, dyds\,,
\end{multline*}
where in the last step we have used the definition of $H_\eps$, and that it is a null solution of $L^*$ in 
$\RR^{n+1}_-$.  We now obtain \eqref{eq5.38} by letting $\eta \to 0$.  

This finishes the argument.
\ep

\begin{corollary}\label{c5.39} Assume that $L$ is an elliptic operator with $t$-independent coefficients, and that $L$ and $L^*$ satisfy the De Giorgi-Nash-Moser condition. If for some $1<p<2+\eps$ the Dirichlet problem \eqref{Dp'} for $L^*$ is solvable in $\repm$ and, in addition to \eqref{eq1.4}, the solution satisfies the square function estimates \eqref{eq1.12},
then the Regularity problem \eqref{Rp} for $L$ is solvable in $\repm$. That is, assertion (a) implies (b) in Theorem~\ref{t1.11}. 
\end{corollary}

\bp The combination of Theorems~\ref{t5.32} and \ref{t3.1} entails that every solution to $Lu=0$ in $\reu$, $u\Bigl|_{\partial\reu}=f\in C_0^\infty(\rn)$ in the weak sense of Lemma~\ref{l2.4} satisfies the estimate \eqref{eq1.5}. 
\ep

\begin{lemma}\label{l5.43} Assertions (b) and (f) of Theorem~\ref{t1.11} are equivalent. \end{lemma}

\bp  The fact that assertion $(b)$ of Theorem~\ref{t1.11} implies $(f)$ was proved in \cite{BM}, Remark~9.15. Indeed, our notion of solvability of \eqref{Rp} as defined in the Introduction, complemented by the Remark following Lemma~\ref{l2.2}, implies compatible solvability of problem $(D)_{p,1}$ as defined in \cite{BM} (see definitions in Section~2.4 of \cite{BM}, in particular, Definition~2.37 and the definition of compatible solvability right after it). Note, in particular, our comments regarding the sense of the convergence to the boundary data in the introduction to the present manuscript. 

Then Remark~9.15 in \cite{BM} detailed the passage from compatible solvability to the well-posedness. 

The fact that $(f)$ implies $(b)$ in Theorem~\ref{t1.11} follows directly from the definitions.\ep

\section{Invertibility of layer potentials. Layer potential representations of solutions}

Propositions~\ref{p5.1} and Remark \ref{r5.11} ascertain that the layer potentials are always bounded for elliptic operators with $t$-independent coefficients that satisfy the De Giorgi-Nash-Moser bounds. We note that, furthermore, \eqref{eq5.5} and Lemma~\ref{l2.2} imply that for $f\in L^p(\rn)$, $1<p<2+\eps$, the single layer potential $S^Lf$ has well-defined normal and tangential derivatives on the boundary in the sense of Lemma~\ref{l2.2} (denoted below by $-\vec e_n\cdot A\,\nabla S^Lf$ and $\nabla_\|S^Lf$, respectively). Both $\vec e_n\cdot A\,\nabla S^Lf$ and $\nabla_\|Sf$ belong to $L^p$ with the appropriate estimates. In particular, 
\begin{equation}\label{eq6.1}
S^L:L^p(\rn)\to\dot L^p_1(\rn),
\end{equation}
\noindent is a bounded operator for all $1<p<2+\eps$. When necessary, we shall use the superscripts $+$ and $-$ to underline the limit taken from the upper and the lower half space, respectively.

Recall the following jump relations:
\begin{align}
\label{eq6.2}
\vec e_n\cdot A(\nabla S^L)^+g - \vec e_n\cdot A(\nabla S^L)^-g&= -g
,\\
\label{eq6.3}
(\nabla_\parallel S^L)^+ g - (\nabla_\parallel S^L)^- g &= 0,\\
\label{eq6.4}
(D^L)^+ f - (D^L)^- f &= -f
.\end{align}

The jump relations \eqref{eq6.2}, \eqref{eq6.3} can be found, e.g.,  in the proof of \cite[Lemma 4.18]{AAAHK} for $g\in L^2(\RR^n)$, and then extended to any $g\in L^p(\RR^n)$, $1<p<2+\eps$, using \eqref{eq5.5} and Lemma~\ref{l2.2}. In particular $(S^L)^+=(S^L)^-$ regarded as operators $L^p(\rn)\to\dot L^p_1(\rn)$,  $1<p<2+\eps$. Similarly, \eqref{eq6.4} is established for $f\in C_0^\infty(\rn)$ in \cite[Lemma 4.18]{AAAHK} and can be extended by continuity to $f\in \dot L^p_1(\rn)$ for all $p$ such that \eqref{eq5.21} holds.

\begin{proposition} \label{p6.5}
Assume that $L$ is an elliptic operator with $t$-independent coefficients, and that $L$ and $L^*$ satisfy the De Giorgi-Nash-Moser condition. Assume further that for some $1<p<\infty$ the Regularity problem \eqref{Rp} for the operator $L$ is solvable in $\repm$.
Then the operator $(S^L)^\pm: L^p(\rn)\to \dot L^p_1(\rn)$  is compatibly invertible, $\frac 1p+\frac{1}{p'}=1$. An analogous statement holds for $L^*$.
\end{proposition}

A few remarks are in order here.

\begin{remark}\label{r6.5.2} We note that in the statement of Proposition~\ref{p6.5} it is sufficient to assume instead of solvability of \eqref{Rp} that the Dirichlet problem \eqref{Dp'} for $L^*$ is solvable in $\repm$ and, in addition to \eqref{eq1.4}, the solution satisfies the square function estimates \eqref{eq1.12}. This would entail \eqref{Rp} by Corollary~\ref{c5.39}.
\end{remark}

\begin{remark}\label{r6.5.3} Finally, let us point out that the solvability of \eqref{Rp} for some $1<p<2+\eps$ implies that for every $f\in C_0^\infty(\rn)$ the weak solution to $Lu=0$ in $\reu$, $u\Bigl|_{\partial\reu}=f\in C_0^\infty(\rn)$ has a variational conormal in $L^p(\rn)$ which satisfies estimate \eqref{eq3.2}. This is a consequence of the definition of \eqref{Rp} and Lemma~\ref{l2.2}, \eqref{eq2.3}.
\end{remark}

\noindent {\it Proof of Proposition~\ref{p6.5}}. First of, the boundedness of $(S^L)^\pm$ follows from the discussion of \eqref{eq6.1} above. The bounds from below can be established as follows. 
By \eqref{eq6.2}, \eqref{eq6.3} and Lemma~\ref{l2.2}, for any  $g\in C_0^\infty(\rn)$ we have
\begin{equation}\label{eq6.6}\|g\|_{L^p(\rn)}
=
\|\vec e_n\cdot A(\nabla S^L)^+g - \vec e_n\cdot A(\nabla S^L)^-g\|_{L^p(\rn)}
 \leq C\,\|(\nabla_\| S^Lg)^{\pm}\|_{L^p(\rn)}.
\end{equation}
To justify this, we point out that whenever $g\in C_0^\infty(\rn)$, the functions $(S_L)^{\pm}g$ belong
to the space $L_1^2(\rn)\cap L^r(\rn)\cap L^q(\rn)$ from Lemma~\ref{l2.4} and to $\dot L_1^p(\rn)$. This can be seen from \eqref{eq6.1} and its dual. 
Hence, $(S_L)^{\pm}g$ can be approximated by $f_k\in C_0^\infty(\rn)$ in $L_1^2(\rn)\cap L^r(\rn)\cap L^q(\rn)$ and in $\dot L_1^p(\rn)$. 
Using Lemma~2.5 in \cite{AAAHK}, one can also show that the solution $S_t g$ belongs to $\widetilde W^{1,2}(\reu)$. 
Now, for each such $f_k$ there exists a weak solution of the Dirichlet problem by Lemma~\ref{l2.4} (call it $u_k$) and, due to \eqref{eq2.5}, these solutions converge to $S_t g$ in $\widetilde W^{1,2}(\rn)$ norm. Moreover, $u_k$ satisfy estimate \eqref{eq3.2} by our assumptions and Remark~\ref{r6.5.3}. Hence, the corresponding weak conormal derivatives $\partial_{\nu_A}u_k$ form a Cauchy sequence in $L^p$ and thus, converge to some $L^p$ function, $h$, in $L^p$ norm. On the other hand,  as mentioned above, $u_k$ converge to $S_t g$ in $\widetilde W^{1,2}(\rn)$ norm and hence, $\partial_{\nu_A}u_k$ converge to the conormal derivative of $S_t g$ in the sense of distributions. Thus, the conormal derivative of $S_t g$ can be identified with $h$ and its $L^p$ norm is bounded by the $L^p$ norm of the tangential derivative, as desired. 

Next, we note that \eqref{eq6.6} for all $g\in C_0^\infty(\rn)$ implies \eqref{eq6.6} for all $g\in L^p(\rn)$. Indeed, every $g\in L^p(\rn)$ can be approximated in $L^p$ by $C_0^\infty$ functions $g_k$, and by \eqref{eq6.1} $(\nabla_\| S^Lg)^{\pm}$ is well-defined, belongs to $L^p$, and the sequence $(\nabla_\| S^Lg_k)^{\pm}$ converges in $L^p$ to $(\nabla_\| S^Lg)^{\pm}$. Thus, \eqref{eq6.6}  holds for all $g\in L^p(\rn)$ as well.

Let us show that $(S^L)^\pm$ is surjective\footnote{We note that the surjectivity argument presented here was inspired by the surjectivity result for layer potentials in \cite{BM}}. Choose some $f\in C_0^\infty(\rn)$. Let $u^\pm$ be the solutions to the Dirichlet problem in the weak sense of Lemma~\ref{l2.4} with boundary data~$f$ in $\RR^{n+1}_\pm$. By our assumptions and Remark~\ref{r6.5.3}, the distributions $g^\pm = \vec e_n\cdot A\nabla u^\pm\Big|_{\partial\RR^{n+1}_\pm}$ lie in $L^p(\RR^n)$.

We shall show that 
\begin{equation}\label{eq6.7}
(S^L)^\pm (g_--g_+)=f\quad\mbox{in}\quad \dot L^p_1(\RR^n).
\end{equation}
According to the proof of Theorem~\ref{t5.32}, we have 
$$ u^{\pm}=-(D_t^L)^{\pm} (f)+(S_t^L)^{\pm}(g^{\pm}). $$

Using jump relations \eqref{eq6.2}--\eqref{eq6.4}, we deduce that $$u^+\Bigr|_{\rn}-u^-\Bigr|_{\rn}=f+(S^L)^\pm (g_+-g_-).$$ Restriction to the boundary is interpreted here in the sense of the non tangential limit and the weak-$L^p$ limit of the gradient, and the equality holds in $\dot L^p_1(\rn)$. We used, in particular, the fact that $(S^L)^+=(S^L)^-$ regarded as operators $L^p(\rn)\to\dot L^p_1(\rn)$,  alluded to above. On the other hand, by definition of $u$ as a solution to  \eqref{Rp} with data $f$, we have $u^+\Bigr|_{\rn}-u^-\Bigr|_{\rn}=0$, again in the sense of the non-tangential limit of the gradient and equality  in $\dot L^p_1(\rn)$. The combination of these two facts immediately yields \eqref{eq6.7} for $f\in C_0^\infty(\rn)$. Now, for general $f\in \dot L_1^p(\rn)$ we can use the limiting procedure, \eqref{eq6.7} for $C_0^\infty$ functions and \eqref{eq6.6} to find $g\in L^p(\rn)$ such that $(S^L)^\pm g=f$.

Let us discuss the compatibility of the involved inverses.  Note that by construction, if $f\in C_0^\infty(\rn)$ then  $g:=g_--g_+$ such that $(S^L)^\pm g=f$ satisfies $g\in L^p(\rn)$ and also $g\in \dot L^2_{-1/2}(\rn)$. By density we can conclude that the operator $(S^L)^\pm: L^p(\rn)\cap \dot L^2_{-1/2}(\rn) \to \dot L^p_1(\rn)\cap \dot L^2_{1/2}(\rn) $ is also surjective (cf. Theorem~9.13 in \cite{BM}). Thus, the resulting inverse indeed satisfies the compatibility property (see the proof of Theorem~3.18 in \cite{BM} for a detailed discussion). \ep

\begin{corollary}\label{c6.8} Assume that $L$ is an elliptic operator with $t$-independent coefficients, and that $L$ and $L^*$ satisfy the De Giorgi-Nash-Moser condition. 

If \eqref{Rp} is solvable in $\repm$ for some $1<p<2+\eps$ then the solution can be represented by means of compatible  layer potentials, that is, assertion (b) implies (c) in Theorem~\ref{t1.11}. An analogous statement holds for $L^*$.
\end{corollary}

\bp The corollary is a combination of the boundedness results in Proposition~\ref{p5.1} and compatible invertibility of layer potentials following from Proposition~\ref{p6.5}.
We only point out that, in the presence of compatibility,  the solution with $C_0^\infty$ data furnished by layer potentials is indeed the weak solution, thus complying with our definitions. Various versions of this fact have already been used above, but let us repeat the argument. Indeed, due to the fact that by our assumptions for every $f\in C_0^\infty(\rn)$ the function $u(x,t)=S_t^{L^*} \left(S_0^{L^*}\right)^{-1}f(x),$ $(x,t)\in \RR^{n+1}_\pm$, furnishes a solution and that the involved inverses are compatible, we see that $\left(S_0^{L^*}\right)^{-1}f \in \dot L_{1/2}^2(\rn)$ and thus,
\begin{equation}\label{eq6.14}
\nabla S_t^{L^*} \left(S_0^{L^*}\right)^{-1}f \in L^2(\reu)
\end{equation}
(see, e.g., \cite{BM}, Theorem~3.1). With a little more care one can show that in fact, $S_t^{L^*} \left(S_0^{L^*}\right)^{-1}f \in \widetilde W^{1,2}(\reu)$, either by a direct computation, or, alternatively, using \eqref{eq6.14} and the procedure described in the proof of Lemma~\ref{l2.4}.
\ep

\begin{corollary}\label{c6.9} Assume that $L$ is an elliptic operator with $t$-independent coefficients, and that $L$ and $L^*$ satisfy the De Giorgi-Nash-Moser condition. 

If \eqref{Rp} is solvable in $\repm$ for $L$ for some $1<p<2+\eps$ with the solution represented by means of compatible layer potentials then \eqref{Dp'} is solvable in $\repm$ for $L^*$, with the solution represented by means of  compatible layer potentials, that is, assertion (c) implies (d) in Theorem~\ref{t1.11}. An analogous statement holds for $L^*$.
\end{corollary}

\bp  We note first that for any fixed $t>0$ the operator $S^{L^*}_t$ is the Hermitian adjoint of $S^{L}_t$
(hence, in particular, it is bounded $(S^{L^*}_t)^\pm: L^{p'}_{-1}(\rn)\to  L^{p'}(\rn)$ uniformly in $t$ -- the fact that one can also deduce from \eqref{eq5.16}). Moreover, there exists a bounded operator $(S^{L^*})^\pm: L^{p'}_{-1}(\rn)\to  L^{p'}(\rn)$, such that for every fixed $f\in L^{p'}(\rn)$ the sequence $(S^{L^*}_t)^\pm f$ converges to $(S^{L^*})^\pm f$ weakly in $L^{p'}$ and finally, the operator $S^{L^*}$ is the Hermitian adjoint of $S^{L}$. All this is detailed in \cite{HMiMo}, \cite{HMM}. 

In the assumptions of the present Corollary, the operator $(S^L)^\pm: L^p(\rn)\to \dot L^p_1(\rn)$  is bounded and invertible. It follows that the operator 
$$(S^{L^*})^\pm: L^{p'}_{-1}(\rn)\to  L^{p'}(\rn)$$ is bounded and invertible.  Moreover, compatibility of inverses would be inherited by the dual. Taking in account \eqref{eq5.16},  the solvability of the Dirichlet problem via the representation \eqref{eq1.15} essentially follows immediately once we observe that $h \in L_{-1}^{p'}$ if and only if there exists $\vec H \in L^{p'}$ such that $\div_{\|} \vec H=h$, with the accompanying norm equivalence.  Note that we already showed in the proof of Corollary~\ref{c6.8} that, in the presence of compatibility,  the solution with $C_0^\infty$ data furnished by layer potentials is indeed the weak solution, thus complying with our definitions.
\ep

\begin{proposition}\label{p6.13}
Assume that $L$ is an elliptic operator with $t$-independent coefficients, and that $L$ and $L^*$ satisfy the De Giorgi-Nash-Moser condition. 

If \eqref{Dp'} is solvable in $\repm$ for $L^*$ for some $1<p<2+\eps$ with the solution represented by means of compatible layer potentials then \eqref{Dp'} is well-posed in $\repm$ for $L^*$, that is, assertion (d) implies (e) in Theorem~\ref{t1.11}. An analogous statement holds for $L$.
\end{proposition}

\bp First, we show that the solvability of \eqref{Dp'} with compatible layer potential representations  implies that for every $f\in L^{p'}(\rn)$ there exists a solution to the Dirichlet problem \eqref{Dp'} with the appropriate convergence to the boundary data. For now, as in the proof of Corollary~\ref{c6.9}, we have only claimed convergence weakly in $L^{p'}(\rn)$.

Going further, we show that the solution converges to the boundary data strongly in $L^{p'}(\rn)$. To this end, take first $f\in C_0^\infty(\rn)$ and recall that $(S_t^{L^*}\nabla_\|)$ is bounded in $L^{p'}(\rn)$ uniformly in $t$ by \eqref{eq5.16}. Then for every $g\in L^p(\rn)$, $t>0$, we have 
\begin{multline} \int_{\rn} (S_t^{L^*}\nabla_\|-S^{L^*}\nabla_\|) f(x)\,g(x) \, dx=\int_{\rn} {\rm div}_{\|} f(x)\,(S_t^{L}-S^L)g(x) \, dx\\[4pt]
\lesssim \int_{\rn} {\rm div}_{\|} f(x)\,t \N(\nabla (S_t^{L}-S^L))g(x) \, dx \lesssim t  \| {\rm div}_{\|} f\|_{L^{p'}(\rn)} \|\N(\nabla (S_t^{L}-S^L))g\|_{L^p(\rn)} \\[4pt]   \lesssim t  \| {\rm div}_{\|} f\|_{L^{p'}(\rn)} \|g\|_{L^p(\rn)},
\end{multline}
by Lemma~\ref{l2.2} and \eqref{eq5.5}. It follows that 
\begin{multline*}
\|(S_t^{L^*}\nabla_\|-S^{L^*}\nabla_\|) f\|_{L^{p'}} =\sup_{g\in L^p:\,\|g\|_{L^p}=1} \int_{\rn} (S_t^{L^*}\nabla_\|-S^{L^*}\nabla_\|) f(x)\,g(x) \, dx \\[4pt] \lesssim \sup_{g\in L^p:\,\|g\|_{L^p}=1}\left(t  \| {\rm div}_{\|} f\|_{L^{p'}(\rn)} \|g\|_{L^p(\rn)}\right) \leq C_f t.
\end{multline*}
Hence, $\|(S_t^{L^*}\nabla_\|-S^{L^*}\nabla_\|) f\|_{L^{p'}}$ converges to 0 as $t\to 0$. One concludes that for every fixed  $f\in C_0^\infty(\rn)$ the sequence $(S_t^{L^*}\nabla_\|)f$ converges to its boundary data strongly in $L^{p'}$. Given the uniform in $t\geq 0$ bounds on the operator $(S_t^{L^*}\nabla_\|)$, it follows that for every $f\in L^{p'}(\rn)$ the sequence $(S_t^{L^*}\nabla_\|)f$ converges to its boundary data strongly in $L^{p'}$, as desired. 

The non-tangential and square function estimates for the resulting solution follow from \eqref{eq5.16} and \eqref{eq5.16-bis}, respectively.

Finally, we recall that uniqueness follows from \cite{HMM}. Indeed, the solvability of \eqref{Dp'} with compatible layer potential representations implies, in particular, solvability (and even well-posedness) of \eqref{Rp} by Corollary~\ref{c5.39} and Lemma~\ref{l5.43}. This, in turn, implies solvability of \eqref{Rp} as defined in \cite{HMM} and hence, we have uniqueness for \eqref{Dp'} by \cite{HMM}, Proposition~8.19. This finishes the proof of assertion (e) in Theorem~\ref{t1.11}. \ep

\vskip 0.08 in 
\noindent {\it Proof of Theorem~\ref{t1.11}.} Let us finally combine the results. Under the assumptions of Theorem~\ref{t1.11}, we have $(a)\Longrightarrow (b)$ (Corollary~\ref{c5.39}), $(b)\Longrightarrow (c)$ (Corollary~\ref{c6.8}), $(c)\Longrightarrow (d)$ (Corollary~\ref{c6.9}), $(d)\Longrightarrow (e)$ (Proposition~\ref{p6.13}), $(e)\Longrightarrow (a)$ (by definition).

Finally, $(b)\Longleftrightarrow (f)$ (Lemma~\ref{l5.43}). \ep

\section{Proofs of Corollaries~\ref{c1.16}--\ref{c1.19} and further remarks}\label{sPrCor}

\noindent {\it Proof of Corollary~\ref{c1.16}.} The De Giorgi-Nash-Moser bounds are stable under $L^\infty$ perturbation of the coefficients. Thus, if $L_0$ falls under the scope of Theorem~\ref{t1.11}, so that $L_0$ and $L_0^*$ both have the De Giorgi-Nash-Moser property, then so do $L$ and $L^*$. Thus, the conditions $(a)-(f)$ of Theorem~\ref{t1.11} are equivalent for the operator $L$ too, and it is sufficient to prove one of them. It is easier to access $(c)$ or $(d)$. Let us focus on (c). 

The fact that  $f\mapsto \N(\nabla S_{t}^Lf)$ is bounded in $L^p(\rn)$, $1<p<2+\eps$, for any $t$-independent operator $L$ (with some $\eps$ depending on the ellipticity constant) is known -- see the discussion of \eqref{eq5.5}. Respectively, 
$$ S_0^L:L^p(\rn)\to \dot L_1^p(\rn),$$ 
is bounded by Lemma~\ref{l2.2}. The operator norm in both cases depends on  standard constants only. 

As far as invertibility is concerned, 
we have by analytic perturbation theory (see the Appendix for details)
\begin{equation}\label{eqAnalPert} \|\nabla_\|S_0^L-\nabla_\| S_0^{L^0}\|_{L^p(\rn)\to L^p(\rn)}\leq C \|A_0-A\|_{L^\infty(\rn)},
\end{equation} 
and hence, invertibility of $S_0^L$ follows from that of $S_0^{L^0}$ via the Neumann series. 
Moreover, by the same analytic perturbation argument we can establish that 
\begin{equation}\label{eqAnalPert-bis} \|\nabla_\|S_0^L-\nabla_\| S_0^{L^0}\|_{L^p(\rn)\cap \dot L^2_{-1/2}(\rn)\to L^p(\rn)\cap \dot L^2_{-1/2}(\rn)}\leq C \|A_0-A\|_{L^\infty(\rn)},
\end{equation} 
which assures compatible invertibility, as desired. 
\ep

\noindent {\it Proof of Corollary~\ref{c1.19}.} Recall that the De Giorgi-Nash-Moser bounds for solutions inside the domain and at the boundary are always valid for the operators with real coefficients. 

The validity of $(a)$ for an elliptic operator with real, $t$-independent, possibly non-symmetric, coefficients is the main result of \cite{HKMP}. To be precise, note that the $A^\infty$ property of harmonic measure proved in \cite{HKMP} yields solvability of the \eqref{Dp'} problem for some $1<p'<\infty$ (see \cite{K}, Theorem 1.7.3) and that the solvability as defined in \cite{K}, Theorem 1.7.3, (ii), exactly coincides with our notion due to the fact that the {\it classical} solution of \cite{K} is indeed our weak solution (see the construction on p. 5 of \cite{K} and the accompanying discussion). 
Then the validity of all ascertions $(b)-(f)$ of Theorem~\ref{t1.11} for such an operator follows from Theorem~\ref{t1.11}, and the claimed perturbation results follow from Corollary~\ref{c1.16}.  
\ep
Finally, let us make the following remark. 

\vskip 0.08 in \noindent {\it Remark}. Theorem~\ref{t1.11} is stated in terms of the simultaneous solvability of the corresponding boundary value problems in both lower and upper half spaces. This is used when proving layer potential representations of solutions (or, to be more precise, invertibility results on the boundary). However, one can establish that the solvability of the Dirichlet problem implies the solvability of the corresponding Regularity problem working in one selected half-space. Indeed, for $C_0^\infty$ data and the corresponding weak solution one can directly use Theorem~\ref{t3.1}, representation formula \eqref{eq5.34}, and \eqref{eq5.21}, \eqref{eq5.5}, to get the desired bounds.  Theorem~\ref{t3.1} actually ensures the analogue of the Rellich-type estimate in the same half-space as that of solvability of the Dirichlet problem, and the results \eqref{eq5.21}, \eqref{eq5.5} hold both for upper and lower half space independently of any solvability assumptions.

\section{Appendix: Analyticity of $\nabla S_t$}\label{sAppendix}

In this section we present a discussion of \eqref{eqAnalPert}. It is essentially treated within the framework of the analytic perturbation theory in \cite{Ka} (and was already used, e.g., in \cite{AAAHK}, \cite{B}). However, a detailed argument for the particular case at hand does not seem to be available in the literature, and for completeness, we provide it here.

Let $A_0$ be elliptic, $(n+1)\times (n+1)$, $t$-independent, and bounded measurable, 
and let $A_z:= A_0 + z M$, where
$M$ is $(n+1)\times (n+1)$, $t$-independent, and bounded measurable,
with $\|M\|_{L^\infty(\rn)} \leq 1$.
Set $L_0:= -\dv A_0\nabla$, and $L_z:= -\dv A_z \nabla$, and suppose that null solutions of 
$L_0$ satisfy 
the De Giorgi-Nash-Moser bounds.
For $|z|$ small enough, say $|z|<\eps_0$, we have that $L_z$ is also elliptic, and satisfies the De Giorgi-Nash-Moser bounds.
By ellipticity, for $|z|<\eps_0$,
$$\nabla L_z^{-1}\dv : L^2(\ree) \to L^2(\ree)\,,$$
or equivalently,
$$L_z^{-1}: \dot{W}^{-1,2}(\ree)\to \dot{W}^{1,2}(\ree)\,.$$
Moreover, we have the Taylor expansion
$$\nabla  L_z^{-1}\dv = \nabla L_0^{-1} \dv \sum_{k=0}^\infty \left(zM\nabla L_0^{-1}\dv\right)^k,$$
which is convergent, as a mapping from $\dot{W}^{-1,2}(\ree)$ to $\dot{W}^{1,2}(\ree)$,
if $\eps_0$ is small enough.   Therefore,
the mapping $z \to L_z^{-1}$, taking values in the space of bounded operators
from $\dot{W}^{-1,2}(\ree)$ to $\dot{W}^{1,2}(\ree)$, is analytic in a neighborhood of $z=0$.
The same is true for $L_z^*$, so by trace theory, we then have that
$z\to \tr\circ (L_z^*)^{-1}$ is an analytic mapping, taking values in the space of bounded operators 
from $\dot{W}^{-1,2}(\ree)$ to $\dot{H}^{1/2}(\rn)$.  Here $\tr$ denotes the trace operator, on
$\rn\times \{0\}=\partial(\reu)$.   We then define the single layer potential for $L_z$, denoted by
$\s^{L_z}$, 
as the adjoint of the operator $\tr\circ (L_z^*)^{-1}$, so that
$z\to \s^{L_z}$ is an analytic mapping taking values in the space of bounded operators 
from $\dot{H}^{-1/2}(\rn)$ to $\dot{W}^{1,2}(\ree)$.
By trace theory again, we have that $z\to S_t^{L_z}$ is an analytic mapping
taking values in the space of bounded operators 
from $\dot{H}^{-1/2}(\rn)$ to $\dot{H}^{1,2}(\rn)$, where $S_t^{L_z}$ denotes the restriction
of $\s^{L_z}$ to the hyperplane $x_{n+1} = t$; i.e., 
$z\to \nabla_\|S_t^{L_z}$ is an analytic mapping
taking values in the space of bounded operators 
from $\dot{H}^{-1/2}(\rn)$ to $\dot{H}^{-1,2}(\rn)$.
Thus,
$$z\to \langle \nabla_\|S_t^{L_z} f, g\rangle$$
is an analytic function, for every $f\in\dot{H}^{-1/2}(\rn)$, and every $g \in\dot{H}^{1/2}(\rn)$,
in particular, for every pair $f,g \in C^\infty_0(\rn)$.
By \cite[Theorem 1.1]{HMiMo}, we have that 
$$\sup_{|z|<\eps_0}\sup_{t}\| \nabla_\|S_t^{L_z}\|_{L^p(\rn)\to L^p(\rn)}  \leq C_p \,,\quad 
1<p<2+\epsilon\,,$$
with $C_p$ depending only on $p$, ellipticity, dimension, and the the De Giorgi-Nash-Moser constants.
Therefore by \cite[p. 365]{Ka},
since $C^\infty_0$ is dense in $L^p$ and $L^{p'}$,  we have that for each fixed $t$,
$z\to \nabla_\|S_t^{L_z}$ is an analytic mapping, taking values in the space of bounded
operators on $L^p(\rn)$, in the disk $|z|<\eps_0$.  This means that
$$\sup_t\big\|\frac{d}{dz}\,\nabla_\|S_t^{L_z}\big\|_{L^p(\rn)\to L^p(\rn)}  \leq C_p\,,$$
so that
\begin{equation}\label{eq1*}
\sup_t\|\nabla_\|S_t^{L_z} -\nabla_\|S_t^{L_0}\|_{L^p(\rn)\to L^p(\rn)}
\leq C_p |z|\,.
\end{equation}
Now, given $A_0$ as above, take $M:= (A_1-A_0)/\|A_1-A_0\|_{L^\infty(\rn)}$, so that
$A_1= A_0 +z_1 M$, with $z_1= \|A_1-A_0\|_{L^\infty(\rn)}$.  If $\|A_1-A_0\|_{L^\infty(\rn)}
<\eps_0$, by \eqref{eq1*}, we have that
$$\sup_t\|\nabla_\|S_t^{L_1} -\nabla_\|S_t^{L_0}\|_{L^p(\rn)\to L^p(\rn)}
\leq\, C_p\,\|A_1-A_0\|_{L^\infty(\rn)}\,,$$
as desired.

\end{document}